\documentclass[11pt]{amsart}
\usepackage{geometry}                
\usepackage{graphicx}

\usepackage{amsfonts}
\usepackage{amssymb}
\usepackage{amsmath}
\usepackage{amsthm}
\usepackage{hyperref}
\usepackage{mathrsfs}
\usepackage{epstopdf}
\usepackage{url}
\usepackage[small,nohug, head=curlyvee]{diagrams}
\diagramstyle[labelstyle=\scriptstyle]
\newarrow{Mapsto} |--->

\theoremstyle{plain}
\newtheorem{theorem}{Theorem}[section]
\newtheorem{lemma}[theorem]{Lemma}
\newtheorem{proposition}[theorem]{Proposition}
\newtheorem{corollary}[theorem]{Corollary}
\theoremstyle{remark}
\newtheorem{remark}[theorem]{Remark}
\theoremstyle{definition}
\newtheorem{definition}[theorem]{Definition}
\theoremstyle{definition}

\theoremstyle{remark}

\DeclareMathOperator{\spec}{Spec}
\DeclareMathOperator{\proj}{Proj}

\DeclareMathOperator{\pic}{Pic}

\DeclareMathOperator{\NS}{NS}

\DeclareMathOperator{\effcurve}{NE}
\DeclareMathOperator{\nef}{\overline{\mathcal{K}}}
\DeclareMathOperator{\effdivisor}{Eff}

\DeclareMathOperator{\mov}{Mov}
\DeclareMathOperator{\Cl}{Cl}

\DeclareMathOperator{\Hom}{Hom}
\DeclareMathOperator{\ann}{ann}
\newcommand{\sat}{\mathrm{sat}}
\DeclareMathOperator{\mult}{mult}

\newcommand{\pr}{\mathrm{pr}}

\newcommand{\Id}{\mathrm{Id}}
\DeclareMathOperator{\cone}{Cone}
\DeclareMathOperator{\conv}{Conv}

\DeclareMathOperator{\Fuk}{Fuk}

\newcommand{\gp}{\mathrm{gp}}

\renewcommand{\H}{\mathrm{H}}

\newcommand{\ev}{\mathrm{ev}}

\newcommand{\D}{\mathrm{D}}
\newcommand{\R}{\mathrm{R}}
\newcommand{\LF}{\mathbb{L}}
\newcommand{\Sch}{\mathrm{Sch}}

\makeatletter

\newcommand{\Rmnum}[1]{\expandafter\@slowromancap\romannumeral #1@}
\makeatother

\title{Natural Compactification of the Moduli of Toric Pairs from the Perspective of Mirror Symmetry}

\begin{document}
\author{Yuecheng Zhu} \email{\url{yuechengzhu@math.utexas.edu}} \address{Department of Mathematics, the University of Texas at Austin, 2515 Speedway Stop C1200, Austin, TX, 78712, USA}

\begin{abstract}
We construct a compactification of the moduli of toric pairs by using ideas from mirror symmetry. The secondary fan $\Sigma(Q)$ is used in \cite{Alex02} to parametrize degenerations of toric pairs. It is also used in \cite{CLS} to control the variation of GIT. We verify the prediction of mirror symmetry that $\Sigma(Q)$ for the moduli of toric pairs is equal to the Mori fan of the relative minimal models of the mirror family. As a result, we give an explicit construction of the compactification $\mathscr{T}_Q$ of the moduli of toric pairs which is the normalization of the compactification in \cite{Alex02} and \cite{ols08}.
\end{abstract}
\maketitle
\tableofcontents

We study the compactification problem of the moduli of toric pairs. Fix a polarized toric variety $(X_Q,\mathcal{L})$ obtained from a lattice polytope $Q$, we consider the moduli of divisors in the linear system of $\mathcal{L}$ that do not contain any torus orbit. Following \cite{Alex02}, we call it the moduli of toric pairs. The compactification of this moduli space has been constructed in \cite{Alex02} and \cite{ols08}. The answers provided by these two papers are complete and satisfying: the natural compactification $\mathscr{K}_Q$ is obtained by adding stable toric pairs with appropriate log structures. The normalization of the coarse moduli space is the toric variety $X_{\Sigma(Q)}$ obtained from the secondary fan $\Sigma(Q)$ of $Q$. On the other hand, the secondary fan $\Sigma(Q)$ also appears in the variation of geometric invariant theory(GIT) in (\cite{CLS} Chapter 14 \& 15). In this paper, we show that the connection between these two stories is mirror symmetry: The secondary fan for the moduli of toric pairs is equal to the Mori fan of the relative minimal models of the mirror family (Theorem~\ref{Mori fan for GKZ}). In addition, by using the detailed study of the variation of GIT in \cite{CLS}, we give a more explicit construction of the families over the compactification $\mathscr{T}_Q$ of the moduli of toric pairs (Theorem~\ref{compactification for GKZ}).

Since the compactification problem of the moduli of toric pairs has been solved by \cite{Alex02} and \cite{ols08}, why another paper on it? This paper serves two purposes. First, it is interesting to relate the geometry of the two sides of mirror symmetry. Secondly, just like \cite{Alex02} and \cite{ols08}, this paper is also the simplest case for a general approach to compactification problems and should be put into a broader context. The same philosophy from mirror symmetry applied to the compactifications of moduli of abelian varieties \cite{zhuav} and the moduli of K3 surfaces \cite{GHKtheta}. We want to include this paper to make the program more complete and to manifest the constructions in the simple case. 

The mirror symmetry we use is learnt from Abouzaid's papers \cite{Ab06} and \cite{Ab09}. The moduli of toric pair $(X_Q, \mathcal{L},\Theta )$ is the Fano side (although toric varieties are in general not Fano). The mirror to $X_Q$ is a Landau--Ginzburg model, the dual algebraic torus $T_X$ plus a superpotential $W$. The polarization $\mathcal{L}$ specifies a $1$-parameter degeneration of $W$. We will recall the construction of the mirror in Section~\ref{GKZ mirror}. Note that the proofs of mirror symmetry in \cite{Ab06} and \cite{Ab09} require that the toric varieties be smooth. However, the construction of the mirror Landau--Ginzburg model makes sense for any projective toric variety, and our proof does not rely on the proofs in \cite{Ab06}, \cite{Ab09}. 

We now summarize the contents of the paper. In Section~\ref{the cone construction}, we will introduce the cone constructions, which are functors that turn affine structures to linear structures or monoids. The best--known example is the embedding of an affine space to a vector space of $1$-dimension higher as the hyperplane of height $1$. Our constructions are canonical and don't involve the choices of the embeddings. Since we use these constructions a lot, and the general forms are not written down anywhere, we include the constructions here for future references. In Section~\ref{families}, we will introduce the constructions of degenerating toric varieties over toric bases. These are the simplest cases of toric degenerations. However, to make the paper self--contained, we also include this part here. In Section~\ref{secondary fan}, we will prove Theorem~\ref{Mori fan for GKZ} which says that the secondary fan is the Mori fan of the relative minimal models of the mirror family. In Section~\ref{glue GKZ families}, we will give an explicit construction of the compactification $\mathscr{T}_Q$ of the moduli of toric pairs by stable toric pairs. The main theorem is Theorem~\ref{compactification for GKZ}. The glues of families rely on the study of the wall-crossings for the mirror family in (\cite{CLS} 15.3). 

\subsection*{Acknowledgements}
This paper is part of the author's thesis. The author would like to thank the advisor, Sean Keel for suggesting this project, and for all the help and support. The work was done partially while the author was visiting the Institute for Mathematical Sciences, National University of Singapore in 2014. The author also thanks the institute for the support of the stay. 

\section{The Cone Constructions}\label{the cone construction}
Fix the notations. If $X$ is an $A$-module, and $B$ is an $A$-algebra, the base change $X\otimes_A B$ is denoted by $X_B$.

Given a vector space $V$, one can forget about the zero vector and get an affine space. This is the forgetful functor (from linear to affine) $\mathbb{R}$. The left adjoint to $\mathbb{R}$ is the embedding of an affine space into a vector space of $1$-dimension higher as the hyperplane of height $1$. In this section, we introduce a family of similar left adjoints $\mathbb{L}$, $S$, $C$, which turn affine structures into additive (or linear) structures. 
\begin{definition}
Let $k$ be a field, and $\overline{V}$ be an affine space over some $k$-vector space $V$. Define a $k$-linear structure over the set
\[
\mathbb{L}(\overline{V}):=\big\{(q,t): q\in \overline{V}, t\in k\backslash\{0\}\big\}\cup \big\{(v,0): v \in V\big\}.
\]

The addition is defined to be
\begin{align*}
(p,t)+(q,s)&=\bigg(\frac{t}{t+s}p +\frac{s}{t+s}q, t+s\bigg), \text{ if } t\neq 0, s\neq 0, t+s\neq 0,\\
(p,t)+(v,0)&=\bigg(\frac{1}{t}v+p,t\bigg),\\
(p,t)+(q,s)&=(t(p-q),0), \text{ if } t+s=0, t\neq 0,\\
(u,0)+(v,0)&=(u+v,0).
\end{align*}

The scalar multiplication is 
\begin{align*}
s(q, t)&=(q, st), \text{ if } s\ne 0, t\ne 0\\
s(q, t)&=0, \text{ if } s=0, t\ne 0,\\
s(v, t)&=(sv,0), \text{ if } t=0.
\end{align*}
\end{definition}

Define the degree map 
\begin{align*}
\deg: \mathbb{L}(\overline{V})&\longrightarrow k\\
(q,t)&\longmapsto t.
\end{align*}

The kernel of $\deg$ is $V$. We always identify $q\in \overline{V}$ with $(q,1)\in \mathbb{L}(\overline{V})$. Therefore, we regard $\overline{V}$ as the hyperplane of $\mathbb{L}(\overline{V})$ of height $1$. This embedding induces a bijection between the set of linear functions on $\LF(\overline{V})$, and the set of affine functions on $\overline{V}$. For simplicity, $\mathbb{L}(\overline{V})$ is also denoted by $\mathbb{V}_k$ or $\mathbb{V}$. 

 Suppose $k=\mathbf{R}$. Let $X$ be a free abelian group, $\cong\mathbf{Z}^g$, and $\overline{X}$ be an $X$-torsor. For a positive integer $n$, define
\[
\overline{X}(1/n):=\Big\{q\in \overline{X}_\mathbf{Q}:q=\sum_i \frac{a_i}{n}q_i, \text{ for } q_i\in\overline{X}, a_i\in \mathbf{Z}, \sum_i a_i=n\Big\}.
\]

Regard $\overline{X}$ as a subset of $\LF(\overline{X}_\mathbf{R})$ of degree $1$. Define a free abelian group $\mathbb{L}(\overline{X})$ to be the subgroup of $\LF(\overline{X}_\mathbf{R})$ generated by $\overline{X}$,
\[
\LF(\overline{X}):=\Big\{(q,n)\in \overline{X}_{\mathbf{Q}}\times \mathbf{Z}\backslash\{0\}:q\in \overline{X}(1/n)\Big\}\cup (X\times \{0\}).
\]

As an abelian group $\LF(\overline{X})\cong \mathbf{Z}^{g+1}$. It has a grading by $\deg: \LF(\overline{X})\to \mathbf{Z}$. $\LF(\overline{X})$ is also denoted by $\mathbb{X}$.  

If $\varphi$ is an affine function on $\overline{X}_\mathbf{R}$, the linear part is the differential and is denoted by $\D\varphi$,
\[
\D\varphi(p-q)=\varphi(p)-\varphi(q).
\]

The affine function $\varphi$ is called integral if it takes integer values on $\overline{X}$. It implies that $\D\varphi$ is also integral. Denote the set of integral affine functions by $Aff(\overline{X},\mathbf{Z})$. $\mathbb{X}$ is canonically isomorphic to the dual $Aff(\overline{X},\mathbf{Z})^*$ via
\begin{align*}
(q,n)(\varphi)&=n\varphi(q), n\neq 0\\
(q,0)(\varphi)&=\D\varphi(q).
\end{align*}

The isomorphism above maps $q\in \overline{X}$ to the evaluation map $\ev_q\in Aff(\overline{X},\mathbf{Z})^*$ and thus is denoted by $\ev$. 

In both cases, $\LF$ are functors. For example, if $\varphi:\overline{X}\to\overline{X'}$ is affine (or piecewise affine), the corresponding additive (piecewise additive) map is 
\begin{eqnarray*}
\LF(\varphi): \LF(\overline{X})&\longrightarrow & \LF(\overline{X'}),\\
(q,n)&\longmapsto&\left\{\begin{array}{rl} n\varphi(q) & \text{ if } n\neq 0\\
\D\varphi(q)& \text{ if } n=0.
\end{array}\right. 
\end{eqnarray*} 

A polytope $Q\subset \overline{X}_\mathbf{R}$ is the convex hull of finite points in $\overline{X}_\mathbf{R}$ and is always bounded. If all the points can be chosen from the lattice $\overline{X}$, then $Q$ is called a lattice polytope. Let $Q$ be a lattice polytope in $\overline{X}_\mathbf{R}$. Define $Q(\frac{1}{n}):=\{q\in Q\cap \overline{X}(\frac{1}{n})\}$, and $Q(\mathbf{Q}):=\coprod_{n>0}Q(\frac{1}{n})$. Identify $Q(\mathbf{Q})\cup\{0\}$ with a subset $S(Q)\subset \mathbb{X}$.
\begin{definition}
$S(Q)$ is defined to be a graded monoid, whose underlying set is, 
\[
S(Q):=\Big\{(q,n): q\in Q(1/n), n\in \mathbf{N}\backslash\{0\}\Big\}\cup\{0\},
\]
 
with the addition and grading induced from those on $\mathbb{X}$.
\end{definition} 

\begin{remark}
$Q(\mathbf{Z})\cong \deg^{-1}(1)$. If $Q$ is of full dimension, the associated group of $S(Q)$ is $\mathbb{X}$. $S(Q)$ is a toric monoid. 
\end{remark}

This construction can be generalized to an unbounded lattice polyhedron $Q$. A polyhedron is the (not necessarily finite) intersection of closed half spaces. In this case we have to take the closure to include the infinite direction as the degree $0$ part 
\[
S(Q)_0:=\big\{(\alpha,0): \alpha\in X, Q+\alpha\subset Q\big\}.
\]

Define $(\alpha,0)+(q,n)=(\frac{\alpha}{n}+q,n)$ if $n\neq0$. Although $S(Q)$ is not fine anymore, it is finitely generated over $S(Q)_0$.

\begin{remark}
We will use $q\in Q(\mathbf{Q})$ to represent an element of $S(Q)$ if there is no confusion. If we have chosen an origin in $\overline{X}$, we can identify $\mathbb{X}$ with $ \mathbb{Z}\oplus X$.
\end{remark}

Regard $S(Q)$ as a subset of $\mathbb{X}_\mathbf{R}$. The convex hull generated by $S(Q)$ in $\mathbb{X}_\mathbf{R}$ is a cone, and is denoted by $C(Q)$. Since $Q$ is a lattice polyhedron, $C(Q)$ is a rational polyhedral cone. If $Q$ is a lattice polytope, $C(Q)$ is strongly convex. We have
\[
S(Q)=C(Q)\cap \mathbb{X}, \text { and } Q=C(Q)\cap \mathbb{X}_{\mathbf{R},1}.
\]

\begin{definition}
Let $V$ be a $\mathbf{R}$ vector space. For any (piecewise) affine function $\varphi: Q\to V$, we associate a (piecewise) linear function $\tilde{\varphi}:C(Q)\to V$ by 
\[
\tilde{\varphi}(q):=\left\{
\begin{array}{ll}
\deg(q)\varphi(q/\deg(q)) &\text{if } \deg(q)\neq 0\\
\D\varphi(q) &\text{if }\deg(q)=0
\end{array}\right.
\]
\end{definition}

On the other side, given a (piecewise) linear function $\tilde{\varphi}: C(Q)\to V$, the restriction to $Q$ is an (piecewise) affine function. So there is a bijection between the set of linear functions on $C(Q)$ and the set of affine functions on $Q$. Define $\mathbb{R}$ to be the forgetful functor from the category of vector spaces to the category of affine spaces. A special case of the bijection $\varphi\mapsto \tilde{\varphi}$ is
\begin{corollary}
Let $U$ be an affine space, and $V$ a vector space over the same field. There is a natural isomorphism
\[
\Hom(U,\mathbb{R}(V))\cong\Hom(\mathbb{L}(U), V).
\]

In particular, $\mathbb{L}$ is the left adjoint to the forgetful functor $\mathbb{R}$. 
\end{corollary}

\begin{definition}
A piecewise affine (resp. linear) function $\varphi$ (resp. $\tilde{\varphi}$) is called integral if it is integral on every top-dimensional affine (resp. linear) domain. 
\end{definition}

\begin{remark}
A piecewise affine map ${\varphi}$ is integral if and only if $\tilde{\varphi}$ is integral. This integrality  is stronger than saying $\varphi$ takes integral values on each integral point. It also requires the slopes to be integral. 
\end{remark}

\section{Constructions of the Families}\label{families}
\subsection{Toric Constructions}\label{section toric constructions}
In this section, $k$ is an arbitrary commutative noetherian ring. We abuse the terminology ``toric variety" even when it is not over a field. A toric variety of a lattice polytope $\sigma$ is denoted by $X_\sigma$, and toric variety of a fan $\Sigma$ is denoted by $X_\Sigma$. 

Let $P$ be a toric monoid, i.e. $P$ is fine and saturated and the associated group $P^{\gp}$ is torsion free. Assume $\sigma_P:=\conv(P)$ in $P^{\gp}_\mathbf{R}$. $\sigma_P$ is a rational polyhedral cone. Introduce a partial order on $P^{\gp}_\mathbf{R}$.
\begin{definition}
 For $u,v \in P^{\gp}_{\mathbf{R}}$, we say $u$ is $P$-above $v$, and denote it by $u\stackrel{P}{\geqslant} v$, if $u-v\in \sigma_P$. We say $u$ is strictly $P$-above $v$ if in addition, $u-v\in \sigma_P\backslash P^*_\mathbf{R}$. 
 \end{definition}
 
Let $Q$ be a full dimensional lattice polytope of $\overline{X}_\mathbf{R}$. A paving $\mathscr{P}$ of $Q$ is a finite set of polytopes contained in $Q$, such that
\begin{enumerate}
\item For any two elements $\sigma, \tau\in \mathscr{P}$, the intersection $\sigma\cap\tau$ is a proper face of both $\sigma$ and $\tau$. 
\item Any face of a polytope $\sigma\in \mathscr{P}$ is again an element of $\mathscr{P}$.
\item The union $\cup_{\sigma\in\mathscr{P}}\sigma$ is $Q$.
\end{enumerate}

A paving is called a triangulation if all cells are simplices. A paving or a triangulation is called integral if all cells are lattice polytopes. A paving or a triangulation is regular or coherent if it is obtained from the affine regions of some piecewise affine function. Without explicit mention, all the pavings and triangulations will be integral and coherent. 

\begin{definition}
Let $\varphi$ be a piecewise affine function over $Q$ with values in $P^{\gp}_\mathbf{R}$. $\varphi$ is called $\mathscr{P}$-piecewise affine if the region where $\varphi$ is affine gives a paving $\mathscr{P}$, i.e. for each $\sigma\in \mathscr{P}_{\max}$, $\varphi\vert_\sigma$ is an element in $\mathbb{X}^*_\mathbf{R}\otimes P^{\gp}_\mathbf{R}$, and $\varphi\vert_{\sigma}$ is different for different $\sigma\in\mathscr{P}_{\max}$. In this case, we also call $\mathscr{P}$ the paving of $\varphi$. 

$\varphi$ is called integral if all $\varphi\vert_\sigma\in \mathbb{X}^*\otimes P^{\gp}$. The linear space of piecewise affine functions whose paving is coarser than $\mathscr{P}$ is denoted by $PA(\mathscr{P})$ or $PA(\mathscr{P}, P^{\gp}_\mathbf{R})$. The subset of integral functions is denoted by $PA(\mathscr{P},\mathbf{Z})$. 
\end{definition}

\begin{definition}[bending parameters]\label{definition of bending parameter}
For each codimension $1$ cell $\rho\in\mathscr{P}$ contained in maximal cells $\sigma_+,\sigma_-\in \mathscr{P}$, we can write
\[
\varphi\vert_{\sigma_+}-\varphi\vert_{\sigma_-}=n_\rho\otimes p_\rho,
\]

where $n_\rho$ is the unique primitive element that defines $\rho$ and is positive on $\sigma_+$, and $p_\rho\in P^{\gp}_\mathbf{R}$ is called the bending parameter.
\end{definition} 

\begin{definition}
A piecewise affine function $\varphi$ is $P$-convex if for every codimension one cell $\rho\in \mathscr{P}$, $p_\rho\in P$. It is strictly $P$-convex if all $p_\rho\in P\backslash P^*$. 
\end{definition}

If $\varphi$ is affine over $Q$, the bending parameters are all trivial. If $\varphi, \varphi'\in PA(\mathscr{P})$ have bending parameters $\{p_\rho\}$, $\{p_\rho'\}$ respectively, the bending parameters for $\varphi+\varphi'$ is $\{p_\rho+p'_\rho\}$. Therefore, an element in $PA(\mathscr{P})/Aff$ (with values in $P^{\gp}_\mathbf{R}$) is determined by the collection of bending parameters. 

Let $Q\subset \overline{X}_{\mathbf{R}}$ be a full dimensional lattice polytope with integal points $I:=Q(\mathbf{Z})$, and $\mathscr{P}$ be an integral paving. Fix a sharp toric monoid $P$. Let $\varphi: Q\to P^{\gp}_{\mathbf{R}}$ be a $\mathscr{P}$-piecewise affine, $P$-convex, integral map. Define the lattice polyhedron
\[
Q_\varphi:=\Big\{(\alpha,h)\in Q\times P^{\gp}_\mathbf{R}: h\stackrel{P}{\geqslant} \varphi(\alpha)\Big\}.
\]

Define the $k$-algebra $R_\varphi:=k[S(Q_\varphi)]$ to be the semigroup algebra of $S(Q_\varphi)$. Define $S:=\spec k[P]$ to be the spectrum of the semigroup algebra of $P$, and $\mathcal{X}_{\varphi}:=\proj R_\varphi$. Notice that $R_{\varphi,0}=k[P]$. We have $\pi:\mathcal{X}_{\varphi}\to S$. $\pi$ is projective and of finite type. The line bundle $\mathcal{L}:=\mathcal{O}(1)$ is $\pi$-ample. 

The natural inclusion $P\to S(Q_\varphi)$ is an integral morphism of integral monoids. $S(Q_\varphi)^{\gp}=\mathbb{L}(\overline{X}\times P^{\gp})\cong \mathbb{X}\times P^{\gp}$. Under this isomorphism, the morphism $P^{\gp}\to S(Q_\varphi)^{\gp}$ is the homomorphism
\begin{align*}
P^{\gp}&\longrightarrow \mathbb{X}\times P^{\gp},\\
p&\longmapsto (0,p).
\end{align*}

Then the cokernel of $f^\flat: P\to S(Q_\varphi)$ is equal to the image of $S(Q_\varphi)$ in $S(Q_\varphi)^{\gp}/P^{\gp}=\mathbb{X}$, which is $S(Q)$. Denote the projection $S(Q_\varphi)\to S(Q)$ by $b$. Since $P$ is sharp, by (\cite{ols08} Lemma 3.1.32), for any $q\in S(Q)$, there exists a unique element $\tilde{q}\in b^{-1}(q)$ such that $\tilde{q}\stackrel{P}{\leqslant}x$ for all $x\in b^{-1}(q)$. For any $q\in S(Q)$, define $\vartheta_q:=\mathrm{X}^{\tilde{q}}$. By definition, for any $x\in b^{-1}(q)$, $x-\tilde{q}\in P$. Therefore

\begin{proposition}
$R_{\varphi}$ is a free $k[P]$-module generated by $\{\vartheta_q\}$ for all $q\in S(Q)$. In particular, the basis is parametrized by the rational points $Q(\mathbf{Q})\cup\{0\}$.
\end{proposition}

If $\varphi$ is integral, then $\vartheta_q=\mathrm{X}^{(q,\varphi(q),n)}$. The multiplication rule is, for $\alpha\in Q(\frac{1}{n}), \beta\in Q(\frac{1}{m})$, and
\[
\gamma=\frac{n}{n+m}\alpha+\frac{m}{n+m}\beta\in Q(\frac{1}{n+m}),
\]

we have 
\[
\vartheta_\alpha\cdot\vartheta_\beta=\mathrm{X}^{n\varphi(\alpha)+m\varphi(\beta)-(n+m)\varphi(\gamma)}\vartheta_\gamma.
\]

For $\varphi$, define a monoid $S(Q)\rtimes P$. As a set $S(Q)\rtimes P=S(Q)\times P$. The addition is defined by
\[
(\alpha, p)+(\beta,q)=(\alpha+\beta, p+q+n\varphi(\alpha)+m\varphi(\beta)-(n+m)\varphi(\gamma))
\]

The conclusion is, 
\begin{proposition}
As monoids, 
\[
S(Q_\varphi)\cong S(Q)\rtimes P. 
\]
\end{proposition}

Define $\mathbb{T}:=\spec k[\mathbb{X}]$ and $T:=\spec k[X]$. The exact sequence
\[
\begin{diagram}
0&\rTo &X &\rTo &\mathbb{X}&\rTo^\deg&\mathbf{Z}&\rTo 0
\end{diagram}
\]

induces an exact sequence
\[
\begin{diagram}
1&\rTo&\mathbb{G}_m&\rTo&\mathbb{T}&\rTo&T&\rTo&1.
\end{diagram}
\]

The $\mathbb{X}$-grading of $S(Q_\varphi)$ defines a $\mathbb{T}$-action $\varrho$ on $k[S(Q_\varphi)]$. Since the exact sequence splits, $\varrho$ induces a $T$-linearization of $\mathcal{L}$. Any global section $\vartheta\in H^0(\mathcal{X}_\varphi,\mathcal{L})$ is decomposed into eigenvectors of $T$-actions.
\begin{equation}\label{Fourier decomposition for GKZ}
\vartheta:=\sum_{\omega\in Q(\mathbf{Z})}c_\omega\vartheta_\omega,
\end{equation}

for $c_\omega\in k[P]$. Define $\Theta:=(\vartheta)_0$. $\Theta$ is a divisor of $\mathcal{X}_\varphi$, and is flat over $S$. 

Let $S=\spec k[P]$ and $S'=\spec k[P']$ be two affine toric varieties, and $f: S'\to S$ is a toric map induced by a homomorphism $f^\flat:P\to P'$. Then any $P$-convex piecewise affine function $\varphi$ induces a $P'$-convex piecewise affine function $\varphi'$ in the following way. Assume the paving of $\varphi$ is $\mathscr{P}$. $\varphi$ defines a collection of pending parameters $\{p_\rho\}$ for $\mathscr{P}$. Define new bending parameters $\{p_\rho'=f^\flat\circ p_\rho\}$. They determine a piecewise affine function $\varphi'$ whose paving is coarser than $\mathscr{P}$.
\begin{corollary}\label{GKZ functorial direct}
The construction is functorial, i.e., the following diagram is cartesian.
\[
\begin{diagram}
\mathcal{X}_{\varphi'}&\rTo&\mathcal{X}_\varphi\\
\dTo&&\dTo\\
S'&\rTo&S.\\
\end{diagram}
\]
\end{corollary}

\begin{proof}
The morphism $\mathcal{X}_{\varphi'}\to \mathcal{X}_\varphi$ is induced from the base change
\begin{align*}
S(Q)\rtimes P&\to S(Q)\rtimes P'\\
(\alpha, p)&\mapsto (\alpha,f^\flat(p)).
\end{align*}
\end{proof}

\begin{corollary}[Functorial]\label{GKZ functorial}
Let $S$, $S'$, $P$, $P'$ and $f$ be as above. Let $\varphi$ be a $P$-convex piecewise affine function, and $\varphi'$ be a $P'$-convex piecewise affine functions. Assume that $\psi=f^\flat\varphi-\varphi'$ takes values in $(P')^*$, the invertible elements in $P'$, then we still have the pull back diagram
\[
\begin{diagram}
\mathcal{X}_{\varphi'}&\rTo &\mathcal{X}_\varphi\\
\dTo&&\dTo\\
S'&\rTo&S.\\
\end{diagram}
\]
\end{corollary}

\begin{proof}
We only need to check that $\mathcal{X}_{f^\flat\varphi}\cong\mathcal{X}_{\varphi'}$ over $S'$. The isomorphism is given by
\begin{align}
S(Q)\rtimes P' &\to S(Q)\rtimes P'\\
(\alpha,p)&\mapsto (\alpha,p+\deg(\alpha)\psi(\alpha)).
\end{align}
It preserves the multiplication. It is an isomorphism because $\psi$ takes values in $(P')^*$. 
\end{proof}

Now choose $P'=\mathbf{N}$ and $S'=\mathbf{A}^1_k$. For any toric divisor over $S$ with center in the toric boundary of $S$, we have a discrete valuation $v: P\to \mathbf{N}$. Use $v$ as $f^\flat$. By Corollary~\ref{GKZ functorial direct}, the pull--back $f^*\mathcal{X}_\varphi/\mathbf{A}^1$ is the $1$-parameter family $\mathcal{X}_{\varphi'}/\mathbf{A}^1$ defined by the ordinary convex function $\varphi'=v\circ\varphi$. Assume the paving for $\varphi'$ is $\mathscr{P}'$ which is coarser than $\mathscr{P}$. The generic fiber of $\mathcal{X}_{\varphi'}$ is the polarized toric variety $(X_Q,\mathcal{L})$ given by the lattice polytope $Q$. The central fiber of $\mathcal{X}_{\varphi'}$ is $X_{\mathscr{P}'}=\varinjlim_{\sigma\in \mathscr{P}'}X_\sigma$, the gluing of toric varieties $X_\sigma$'s as indicated by the paving $\mathscr{P}'$. Note that $X_{\mathscr{P}'}$ is denoted by $(P,L)[\Delta',1]$ in (\cite{Alex02} Definition 2.4.2, Corollary 2.4.5) for $\Delta'$ the complex obtained from the paving $\mathscr{P}'$.

Consider the section $\vartheta$ in Equation~\eqref{Fourier decomposition for GKZ}. The pull--back of $\Theta$ is stable for $\mathcal{X}_{\varphi'}$ if and only if the coefficients the residue of $c_\omega$ at the origin of $\mathbf{A}^1$ does not vanish for any $\omega\in \mathscr{P}'\cap I$ (\cite{Alex02} Lemma 2.6.1). Therefore, we have

\begin{proposition}\label{stable toric pairs}
Let $(\mathcal{X}_\varphi,\mathcal{L},\varrho)/S$ be constructed as above. If we choose $\vartheta$ such that $c_\omega$ does not vanish at any point $s\in S$ for any $\omega\in Q(\mathbf{Z})\cap\mathscr{P}$, then $(\mathcal{X}, \Theta,\varrho)/S$ is a stable toric pair and is a degeneration of the pair $(X_Q,\Theta)$. 
\end{proposition}


\subsection{The Standard Family}\label{GKZ standard family}
A more logical way would be we define our toric monoids first, and then identify them with monoids appeared in \cite{ols08}. However, in order to reduce notations, we just use notations from \cite{ols08} and show that they can be naturally obtained from the mirror constructions in Section~\ref{secondary fan} and~\ref{glue GKZ families}. Define $N_i=S(\sigma_i)$, and $N_\mathscr{P}:=\varinjlim N_i$ (in the category of integral monoids). And notice that $\Hom(\varinjlim N_i, \mathbf{Z})\cong \varprojlim(\Hom(N_i,\mathbf{Z}))$, $N_\mathscr{P}^{\gp}$ is the dual of $\varprojlim(\Hom(N_i,\mathbf{Z}))$. Also consider the direct limit in the category of sets $S(\underline{Q}):=\varinjlim N_i$. As a set, $S(\underline{Q})=S(Q)$. $\varprojlim(\Hom(N_i,\mathbf{Z}))$ is the group of integral piecewise linear functions on $C(\underline{Q})$, and is equal to $PA(\mathscr{P},\mathbf{Z})$. Define the natural map $\tilde{\varphi}':S(\underline{Q})\to N_\mathscr{P}$ by the universal property of $S(\underline{Q})$. $\tilde{\varphi}'$ should be regarded as the universal $\mathscr{P}$-piecewise linear function because $N_\mathscr{P}^{\gp}=(PA(\mathscr{P},\mathbf{Z}))^*$. However, we only need convex functions and we want to neglect affine functions. Define $H_\mathscr{P}\subset N_\mathscr{P}^{\gp}$ to be the submonoid generated by 
\[
\alpha*\beta:=\tilde{\varphi}'(\alpha)+\tilde{\varphi}'(\beta)-\tilde{\varphi}'(\alpha+\beta), \forall \alpha,\beta\in S(Q).
\]

The evaluation is a natural map from $H_\mathscr{P}$ to the space $PA(\mathscr{P},\mathbf{Z})^*$. This map is injective because everything is integral. The image is in $(PA(\mathscr{P},\mathbf{Z})/Aff)^*$, because $H_\mathscr{P}\subset \ann(Aff)$. Moreover, $H_\mathscr{P}$ is sharp because $\mathscr{P}$ is regular. 

\begin{proposition}\label{log structure}
Let $C(\mathscr{P},\mathbf{Z})$ be the submonoid of the image of integral convex functions in $PA(\mathscr{P},\mathbf{Z})/Aff$, then 
\[
H_\mathscr{P}^\sat=C(\mathscr{P},\mathbf{Z})^\vee
\]
\end{proposition}

\begin{proof}
It is easy to see that $\alpha*\beta\in C(\mathscr{P},\mathbf{Z})^\vee$. It follows that $H_\mathscr{P}\subset C(\mathscr{P},\mathbf{Z})^\vee$. On the other hand, for any $\bar{\psi}\in PA(\mathscr{P},\mathbf{Z})/Aff$. Pick any lift $\psi\in PA(\mathscr{P},\mathbf{Z})$. $\psi$ is a $\mathscr{P}$-piecewise affine function over $\underline{Q}$. $\psi(\alpha*\beta)\geqslant 0$ for all $\alpha,\beta$ if and only if it is convex. Therefore $(H_{\mathscr{P}})^\vee=C(\mathscr{P},\mathbf{Z})$. We claim that $H_\mathscr{P}^{\gp}=\ann(Aff)\cap N_\mathscr{P}^{\gp}$. Assume this, then, since $H_\mathscr{P}$ is sharp, $H_\mathscr{P}^{\sat}=((H_\mathscr{P})^\vee)^\vee=C(\mathscr{P},\mathbf{Z})^\vee$. 

Now we prove the claim. By (\cite{ols08} Lemma 3.1.7), $H_\mathscr{P}$ is equal to the image of $SC_1(\mathbb{X}_{\geqslant 0})/B_1$. By the description of $SC_1(\mathbb{X}_{\geqslant 0})/B_1$ in the proof of (\cite{Alex02} Lemma 2.9.5), it is the semigroup generated by integral points in the kernel that defines all the linear relations in $\mathbb{X}_\mathbf{R}$. Therefore $H_\mathscr{P}^{\gp}$ is saturated in $\ann(Aff)$.  
\end{proof}

Up to a global affine function, the piecewise affine function $\varphi'$ is equal to a piecewise affine function $\varphi: Q\to H_\mathscr{P}^{\gp}$. (Because the bending parameters are all in $H_\mathscr{P}^{\gp}$. We can simply require $\varphi$ to be zero on some top-dimensional cell.) By definition, $\varphi$ is $H_\mathscr{P}$-convex and integral. Let $k=\mathbf{Z}$, and define $S(Q_\varphi)$ and $\mathcal{X}_\mathscr{P}:=\mathcal{X}_\varphi$ as before. This is called the standard family. 

Define $S(Q)\rtimes H_\mathscr{P}^\sat$ to be the monoid with underlying set $S(Q)\times H_\mathscr{P}^\sat$ and the addition given by
\[
(\alpha,p)+(\beta,q):=(\alpha+\beta,p+q+\alpha*\beta).
\]

From the above discussion,  $R_\varphi\cong k[S(Q)\rtimes H_\mathscr{P}^\sat]$ and 
\[
\mathcal{X}_\mathscr{P}\cong\proj k[S(Q)\rtimes H_\mathscr{P}^\sat]\to \spec k[H_\mathscr{P}^\sat].
\]

Therefore, our standard family is the saturation of the standard family in \cite{ols08}. We can use the pull--back log structures from the standard family in \cite{ols08}, and denote them by $(\mathcal{X}_\mathscr{P}, M_\mathscr{P})$ over $(S,M_S)=\spec (H_\mathscr{P}\to k[H_\mathscr{P}^\sat])$. By (loc. cit. 3.1.12), the log morphism $(\mathcal{X}_\mathscr{P}, M_\mathscr{P})\to (S, M_S)$ is integral and log smooth.

\begin{proposition}\label{bending parameters for GKZ}[bending parameters for $\varphi$]
For any codimension-$1$ wall $\rho$, the bending parameter $p_\rho\in H_\mathscr{P}^\sat$ is described as follows. Let $\rho$ be the common face of maximal cells $\sigma_i$ and $\sigma_j$. For any $\mathscr{P}$-piecewise affine function $\psi\in PA(\mathscr{P},\mathbf{R})$, $\psi_i$ and $\psi_j$ are affine extensions of $\psi|_{\sigma_i}$ and $\psi|_{\sigma_j}$. Let $L(\rho)$ be the linear subspace in $\mathbb{X}_\mathbf{R}$ generated by $\rho$. Let $\omega\in \overline{X}$ be a point that maps to a minimal generator of $\mathbb{X}/L(\rho)\cap\mathbb{X}\cong \mathbf{Z}$ and is on the same side as $\sigma_i$, then the bending parameter is defined by
\begin{equation}\label{bending parameters for the standard family in GKZ}
p_\rho(\psi)=\psi_i(\omega)-\psi_j(\omega). 
\end{equation}
\end{proposition}

\begin{proof}
 $\forall \psi\in PA(\mathscr{P},\mathbf{R})$, we have 
\begin{align*}
\tilde{\varphi}(\alpha)(\psi)+\tilde{\varphi}(\beta)(\psi)-\tilde{\varphi}(\alpha+\beta)(\psi)=\psi(\alpha)+\psi(\beta)-\psi(\alpha+\beta),\\
(\psi\circ\tilde{\varphi}-\psi)(\alpha+\beta)=(\psi\circ\tilde{\varphi}-\psi)(\alpha)+(\psi\circ\tilde{\varphi}-\psi)(\beta)\quad \forall \alpha,\beta.
\end{align*}

It means $\psi\circ\tilde{\varphi}-\psi$ is a global linear function on $S(Q)$. It follows that
\[
(\tilde{\varphi}_i-\tilde{\varphi}_j)(\psi)=\psi_i-\psi_j.
\]

Therefore, by definition,
\[
p_\rho(\psi)=\psi_i(\omega)-\psi_j(\omega). 
\]
\end{proof}
 

\section{The Secondary Fan is the Mori fan}\label{secondary fan}
\subsection{The Secondary Fan}\label{GKZ secondary fan}
In this section, we study the secondary fan defined in \cite{GKZ}. Fix $Q\subset \overline{X}_\mathbf{R}$ a lattice polytope of full dimension $g$. Denote $Q(\mathbf{Z})$ by $I$, and assume the cardinality $|I|$ is $N$.



For any function $\psi$ on $I$ and  a triangulation $\mathscr{T}$, we define a $\mathscr{T}$-piecewise affine function $g_{\psi,\mathscr{T}}:Q\to \mathbf{R}$. For each vertex $\omega$ of $\mathscr{T}$, we have $g_{\psi,\mathscr{T}}(\omega)=\psi(\omega)$.  The function $g_{\psi,\mathscr{T}}$ is obtained by affinely interpolating $\psi$ inside each simplex of $\mathscr{T}$. 

We regard each function $\psi$ as an element in $\mathbf{R}^I$. 

\begin{definition}\label{definition of the cone of tiling}
Let $\mathscr{T}$ be a triangulation of $Q$. We shall denote by $\widetilde{C}(\mathscr{T})$ the cone in $\mathbf{R}^I$ consisting of functions $\psi:I\to \mathbf{R}$ with the following two properties: 
\begin{itemize}
\item[a)] The function $g_{\psi, \mathscr{T}}:Q\to \mathbf{R}$ is convex\footnote{A function $f$ is convex if $f(tx+(1-t)y)\leqslant tf(x)+(1-t)f(y)$. Notice that it is different from the usual convention in the literature of toric geometry.}.
\item[b)] For any $\omega\in I$ but not a vertex of any simplex from $\mathscr{T}$, we have $g_{\psi, \mathscr{T}}\leqslant \psi(\omega)$. 
\end{itemize}
\end{definition}

If $\mathscr{T}$ is a coherent triangulation, then the map $\mathbf{R}^I\to PA(\mathscr{T},\mathbf{R})$ which sends $\psi$ to $g_{\psi,\mathscr{T}}$ is a surjection, and the evaluation map $PA(\mathscr{T},\mathbf{R})\to \mathbf{R}^I$ is injective. Thus $\mathbf{R}^I=PA(\mathscr{T},\mathbf{R})\times \mathbf{R}^{I_\emptyset}$, where $I_\emptyset=I\backslash \mathscr{T}$. The projection $\pr: \mathbf{R}^I\to \mathbf{R}^{I_\emptyset}$ is described as follows: For any $\omega\in I_\emptyset$, $\omega$ is in the interior of some simplex $\sigma\in \mathscr{T}$, and assume $\omega=\sum_i a_i\omega_i$, for $\{\omega_i\}$ the set of vertices of $\sigma$. The $\omega$-coordinate of $\pr(\psi)$ is $\psi(\omega)-\sum_i a_i\psi(\omega_i)$. Denote the convex piecewise affine $\mathbf{R}$-valued functions by $CPA(\mathscr{T},\mathbf{R})$, $\widetilde{C}(\mathscr{T})=CPA(\mathscr{T},\mathbf{R})\times \mathbf{R}_{\geqslant 0}^{I_\emptyset}$.  Since $N_{\mathscr{T},\mathbf{R}}^{\gp}=(PA(\mathscr{T},\mathbf{R}))^*$, we also have $(\mathbf{R}^I)^*=N_{\mathscr{T},\mathbf{R}}^{\gp}\times (\mathbf{R}^{I_\emptyset})^*$. We need the following proposition.

\begin{proposition}[\cite{GKZ} Chapter. 7, Proposition. 1.5]\label{complete fan}
Fix $Q$ and $I$. The cones $\widetilde{C}(\mathscr{T})$ for all the coherent triangulations of $Q$ together with all faces of these cones form a complete generalized fan in $\mathbf{R}^I$. 
\end{proposition}

For each $\omega\in I$, let $e_\omega$ denote the evaluation of a function at $\omega$. $\{e_\omega\}$ is a canonical basis for $(\mathbf{Z}^I)^*$. Consider the linear map $p:(\mathbf{Z}^I)^*\to \mathbb{X}$ defined by $e_\omega\mapsto (\omega,1)$. Denote $\ker p$ by $\mathbb{L}$. $\mathbb{L}=Aff^{\perp}$. The dual $\mathbb{L}^*$ is defined to be the quotient of $\mathbf{Z}^I$ by $Aff(\overline{X},\mathbf{Z})$. It is not necessarily torsion free. Let the free part be $\mathbb{L}^*_f$. Define $q:\mathbf{Z}^I\to \mathbb{L}^*_f$. The kernel of $q$ is the saturation of $Aff(\overline{X},\mathbf{Z})$ in $\mathbf{Z}^I$.

\begin{definition}[the secondary fan]\label{definition of the secondary fan}
The image of $\widetilde{C}(\mathscr{T})$ with all their faces form a complete fan $\Sigma(Q)$ in $\mathbb{L}^*_\mathbf{R}=\mathbf{R}^I/Aff$, and is called the secondary fan. The image of $\widetilde{C}(\mathscr{T})$ in $\mathbb{L}^*_\mathbf{R}$ is denoted by $C(\mathscr{T})$. 
\end{definition} 

\begin{remark}
This definition is different from that in \cite{GKZ}. We take the negative of GKZ's fan and then take the quotient. 
\end{remark}

\begin{lemma}
We have
\[
C(\mathscr{T})=C(\mathscr{T},\mathbf{R})\times \mathbf{R}^{I_\emptyset}_{\geqslant 0}, \quad C(\mathscr{T})^\vee=H_{\mathscr{T},\mathbf{R}}^{\sat}\times \mathbf{R}^{I_\emptyset}_{\geqslant 0}. 
\]
\end{lemma}

\begin{remark}
The secondary fan $\Sigma(Q)$ is the normal fan of the secondary polytope $P(Q)$, which is a full-dimensional polytope in some affine space $\overline{\mathbb{L}}_\mathbf{R}$ over $\mathbb{L}_\mathbf{R}$. Therefore the dimension of $P(Q)$ is $N-g-1$. 
\end{remark}


\subsection{The Mirror Interpretation of the Secondary Fan}\label{GKZ mirror}
Let $Y$ be the dual lattice of $X$. Given the lattice polytope $Q\subset\overline{X}_\mathbf{R}$, consider the normal fan $\Delta_Q$. For any face $F\subset Q$, following \cite{fulton}, define the normal cone $\sigma_{F}Q=\{v\in Y_{\mathbf{R}}:\langle u,v\rangle \leqslant \langle u',v\rangle,\:  \forall u\in F, u'\in Q\}$. The collection of cones $\sigma_FQ$, as $Q$ varies over the faces of $Q$, form a complete fan which also defines the toric variety $X_Q=X_{\Delta_Q}$. Let $A$ denote the set of all primitive vectors $v$ of the rays $\Delta_Q(1)$ in the fan $\Delta_Q$, define the lattice polytope $P=\conv\{v\}_{v\in A}$ in $Y_\mathbf{R}$. 

Let $T_X$ be the algebraic torus with the character group $Y$. Regard $P$ as the Newton polytope of the Laurent polynomial
\[
W(z)=-1+\sum_{v\in A}z^v.
\]

For a toric variety $X_{\Delta_Q}$, the toric boundary divisor $D$ is anticanonical. The mirror of the toric pair $(X_{\Delta_Q}, D)$ is defined to be the Landau--Ginzburg model $(T_X, W)$ \cite{Ab06},\cite{Ab09}. The actual form of the potential function $W$ is irrelevant. What really matters is $(T_X, W^{-1}(0))$ as a symplectic pair. If the toric variety $X_{\Delta_Q}$ is smooth, there exists a full sub-pre-category of $\Fuk(T_X, W^{-1}(0))$ which is quasi-equivalent as an $A_\infty$ pre-category to the category of line bundles on $X_{\Delta_Q}$ (\cite{Ab09} Theorem 1.2). 

The precise correspondence of the pre-categories depends on a choice of a degeneration. The polytope $Q$ gives rise to a $1$-dimensional degeneration family of the Landau--Ginzburg models $(\check{\pi}: \overline{\mathcal{Y}}\to \mathbf{A}^1, W_t)$ as follows: The vertices of the polytope $Q$ defines a piecewise linear, strictly convex function $\check{\psi}$ over the fan $\Delta_Q$. Define
\[
W_t(z)=-1+\sum_{v\in A}z^vt^{-\check{\psi}(v)}.
\]

The strict convex piecewise linear function $\check{\psi}$ defines an integral paving $\check{\mathscr{Q}}$ of $P$. Take the discrete Legendre transform $(\overline{X}_\mathbf{R}, \mathscr{Q}, \psi)$ of the triple $(P,\check{\mathscr{Q}},\check{\psi})$. Let $\Pi$ be the non-smooth locus of $\psi$. As $t$ goes to zero, after some re-parametrization, the log amoeba of $W_t^{-1}(0)$ converges to the tropical hypersurface $\Pi$ of $\overline{X}_\mathbf{R}$ (\cite{Mik04} Theorem 5). Notice that by definition of $\check{\psi}$, up to a negative sign, the Legendre transform of the origin inside $P$ is the lattice polytope $Q$. Therefore, if we want to get $\Pi$ as a tropical divisor in the limit, we need to construct a degeneration with $Q$ as the dual intersection complex. Let $\overline{\Sigma}$ be the fan consisting of the faces of the rational polyhedral cone $C(Q)\subset \mathbb{X}_\mathbf{R}$. Let $\overline{\mathcal{Y}}$ denote the affine toric variety $X_{\overline{\Sigma}}$. The morphism $\check{\pi}: \overline{\mathcal{Y}}\to \mathbf{A}^1$ has the generic fiber $T_X$. The family of Landau--Ginzburg models $(\overline{\mathcal{Y}}_t, W_t)$ is the mirror of the triple $(X_Q, D, \mathcal{L})$. 

There is another interpretation of the mirror family $\overline{\mathcal{Y}}\to \mathbf{A}^1$. For each cone $\sigma\in \Delta_Q$, consider the cone of $1$-higher dimension
\[
\tilde{\sigma}=\{(u, t); u\in \sigma, t+\check{\psi}(u)\geqslant 0\}\subset \mathbb{Y}_\mathbf{R}.
\]

Let $\widetilde{\Delta}_Q$ be the fan in $\mathbb{Y}_\mathbf{R}$ consisting of the cones $\tilde{\sigma}$ for $\sigma\in \Delta_Q$ and their faces. This is the fan for the total space of $\mathcal{L}$ (\cite{CLS} Proposition 7.3.1). The primitive vectors in the rays $\widetilde{\Delta}_Q(1)$ are $(v,-\check{\varphi}(v))$. Therefor, the mirror of the total space $\mathcal{L}$ is the $g+1$-dimensional torus $T_\mathbb{X}$ with the superpotential $W(t,z):=W_t(z)$. $\overline{\mathcal{Y}}$ is the partial compactification of $T_{\mathbb{X}}$ with the limit of $W$.

Now we study the resolutions of $\overline{\mathcal{Y}}$ (mirror to the degenerations of the toric pairs $(X_Q, \Theta)$). Let $\mathscr{P}$ be a paving of $Q$. For any cell $\sigma\in \mathscr{P}$, the cone $C(\sigma)$ is a strongly convex rational polyhedral cone in $\mathbb{X}_\mathbf{R}$. The cones $\{C(\sigma)\}$ form a fan denoted by $\Sigma_\mathscr{P}$. Let $\mathcal{Y}_\mathscr{P}$ denote the toric variety $X_{\Sigma_{\mathscr{P}}}$. The natural morphism $f: \mathcal{Y}_\mathscr{P}\to \overline{\mathcal{Y}}$ is birational and proper. By (\cite{EGA3} Theorem 3.2.1), $f_*\mathcal{O}_{\mathcal{Y}_\mathscr{P}}$ is a coherent $\mathcal{O}_{\overline{\mathcal{Y}}}$-module.  Since $\overline{\mathcal{Y}}$ is normal and $f$ is birational, $R_0:=\Gamma(\overline{\mathcal{Y}},\mathcal{O}_{\overline{\mathcal{Y}}})\cong\Gamma(\mathcal{Y}_\mathscr{P}, \mathcal{O}_{\mathcal{Y}_\mathscr{P}})$. If $\mathscr{P}$ is coherent, $\Sigma_\mathscr{P}$ has a strictly convex support function. By (\cite{CLS} Theorem 7.2.4), $f:\mathcal{Y}_\mathscr{P}\to \overline{\mathcal{Y}}$ is projective. Denote the set of rays in $\Sigma_\mathscr{P}$ by $\Sigma_\mathscr{P}(1)$. There is a canonical bijection $\Sigma_\mathscr{P}(1)\to I\cap\mathscr{P}$. 

A projective morphism $f:\mathcal{Y}\to \overline{\mathcal{Y}}$ is a relative minimal model if $K_\mathcal{Y}+Y_0$ is $f$-nef, $\mathcal{Y}$ has terminal singularities and is $\mathbf{Q}$-factorial. In our case, since $\mathcal{Y}_\mathscr{P}$ is a toric variety, $K+Y_{\mathscr{P},0}$ is trivial. By (\cite{CLS} Exercise 8.2.14 (a)), $\mathcal{Y}_\mathscr{P}$ is Gorenstein. By (\cite{CLS} Proposition 11.4.12 (b)), $\mathcal{Y}_\mathscr{P}$ has canonical singularities, and if $I=Q(\mathbf{Z})\subset \mathscr{P}$, $\mathcal{Y}_\mathscr{P}$ has terminal singularities. Moreover, $\mathcal{Y}_\mathscr{P}$ is an orbifold if and only if it is $\mathbf{Q}$-factorial, if and only if $\mathscr{P}$ is a triangulation (\cite{CLS} Proposition 4.2.7 \& Theorem 11.4.8). Therefore, we have 
\begin{lemma}
A toric model $\mathcal{Y}_{\mathscr{P}}\to \overline{\mathcal{Y}}$ is a relative minimal model if $\mathscr{P}$ is a coherent triangulation, and $\Sigma_\mathscr{P}(1)=I$. 
\end{lemma}

If $\mathcal{Y}_\mathscr{P}$ is a relative minimal model, we simply denote $\Sigma_\mathscr{P}(1)$ by $\Sigma(1)$, since it doesn't depend on the decomposition. Any two relative minimal models are $\mathbf{Q}$-factorial, and isomorphic up to codimension $1$. Therefore, they all have the same pseudo-effective cone, and the same Mori fan. We say that the Mori fan is canonical. 

The definition of the Mori fan is taken from \cite{HK}. We modify the definition, and work in the relative case. Assume the morphism between schemes $X\to \overline{X}$ is projective, and $\overline{X}=\spec R_0$ is affine. 

\begin{definition}
For a divisor $D$ on $X$, the section ring is the graded $R_0$-algebra
\[
\R(X,D):=\oplus_{n\in\mathbf{N}} \H^0(X,\mathcal{O}(nD)).
\]
\end{definition}

If $\R(X,D)$ is finitely generated over $R_0$, and $D$ is effective, then there is a rational map over $R_0$,
\[
f_D: X\dashrightarrow \proj \R(X,D).
\]

\begin{definition}
Let $D$, and $D'$ be two $\mathbf{Q}$-Cartier divisors on $X$ with section rings finitely generated over $R_0$. Then we say $D$ and $D'$ are Mori equivalent if the rational maps $f_D$ and $f_{D'}$ have the same Stein factorization. 
\end{definition}

\begin{definition}[the Mori chamber]
Let $X\to \overline{X}$ be as above. Assume that $\R(X,D)$ is finitely generated over $R_0$ for all divisors $D$ on $X$, and $\pic(X)_\mathbf{Q}=\NS_\mathbf{Q}(X)$. By a Mori chamber of $\NS_\mathbf{R}(X)$, we mean the closure of an equivalence class whose interior is open in $\NS_\mathbf{R}(X)$. 
\end{definition}

\begin{remark}
Different Mori chambers are always disjoint.
\end{remark}

\begin{definition}
If all the Mori chambers with their faces form a fan in $\NS_\mathbf{R}(X)$, it is called the Mori fan of $X$. 
\end{definition}

Fix an arbitrary paving $\mathscr{P}$ and consider the toric variety $\mathcal{Y}_\mathscr{P}$. $PA(\mathscr{P},\mathbf{Z})$ is a sublattice of $\mathbf{Z}^I$ of finite index. It might not be saturated. Recall a Cartier divisor for the toric variety $\mathcal{Y}_\mathscr{P}$ is described in terms of an integral piecewise linear function $\psi$ over the fan $\Sigma_\mathscr{P}$. Linear functions correspond to trivial Cartier divisors, and convex functions correspond to nef Cartier divisors. Therefore, $\nef(\mathcal{Y}_\mathscr{P})=C(\mathscr{P},\mathbf{R})$, and $\overline{\effcurve}(\mathcal{Y}_\mathscr{P})=H_{\mathscr{P},\mathbf{R}}^{\sat}$. 

If $I\subset \mathscr{P}$, the exact sequence for the Weil divisor class group is 
\[
\begin{diagram}
0&\rTo& Aff(\overline{X},\mathbf{Z})&\rTo&\mathbf{Z}^{\Sigma(1)}&\rTo& \Cl(\mathcal{Y}_\mathscr{P})&\rTo&0.
\end{diagram}
\]

Recall $\mathbb{X}=Aff(\overline{X},\mathbf{Z})^*$, we have 
\[
\mathbb{L}^*= \Cl(\mathcal{Y}_\mathscr{P}),\quad \mathbb{L}=\Hom(\Cl(\mathcal{Y}_\mathscr{P}),\mathbf{Z}).
\] 

For a $\mathbf{Q}$-factorial toric variety $\mathcal{Y}_\mathscr{P}$,
\[
\Cl(\mathcal{Y}_\mathscr{P})_\mathbf{Q}\cong\NS_\mathbf{Q}(\mathcal{Y}_\mathscr{P})\cong\pic(\mathcal{Y}_\mathscr{P})_\mathbf{Q}.
\]

Therefore if $\mathcal{Y}_\mathscr{P}$ is a relative minimal model, we can identify $\mathbb{L}^*$ with $\NS(\mathcal{Y}_\mathscr{P})$, and say that the Mori fan is fan supported in $\mathbb{L}^*_\mathbf{R}$. The Weil divisor corresponding to the primitive vector $(\omega,1)$ is denoted by $D_\omega$. For a $\mathbf{R}$-Weil divisor $D=\sum_{\omega\in \Sigma(1)} a_\omega D_\omega$, the associated function is defined by $\psi_D(\omega)=a_\omega$. The equivalent class in $\mathbb{L}^*_\mathbf{R}$ is also denoted by $\psi_D$.

\begin{theorem}\label{Mori fan for GKZ}
If $\mathcal{Y}_\mathscr{P}$ is a relative minimal model, the Mori fan exists. It is the secondary fan $\Sigma(Q)$ under the identification above.
\end{theorem}
\begin{proof}
We show that each full dimensional cone $C(\mathscr{T})$ in the secondary fan is a Mori chamber. Fix such a coherent triangulation $\mathscr{T}$. Recall $C(\mathscr{T})=C(\mathscr{T},\mathbf{R})\times \mathbf{R}_{\geqslant 0}^{I_\emptyset}$. For any Cartier divisor $D$ in the interior $C(\mathscr{T})^\circ$, let $\psi_D$ be the associated function in $\mathbf{Z}^I$. $\psi_D=\psi_{E}+\psi_F$ for $\psi_{E}\in C(\mathscr{T},\mathbf{Z})$ and $\psi_F\in \mathbf{Z}_{\geqslant 0}^{I_\emptyset}$. The corresponding divisors are $E$ and $F$. We can multiply $D$ by a positive integer to make both $E$ and $F$ Cartier.

Each Cartier divisor $D=\sum_{\omega\in \Sigma(1)} a_\omega D_\omega$ defines a polyhedron
\[
P_D:=\big\{m\in Aff(\overline{X}_\mathbf{R},\mathbf{R}): \langle m, (\omega,1)\rangle\geqslant -a_\omega, \forall \omega\in \Sigma(1)\big\}.
\] 

By (\cite{CLS} Proposition 4.3.3), 
\[
\Gamma(\mathcal{Y}_\mathscr{P},\mathcal{O}_{\mathcal{Y}_\mathscr{P}}(D))=\bigoplus_{m\in P_D\cap \mathbb{X}^*} k\cdot \mathrm{X}^m.
\]

Notice that for any $\omega\in I_\emptyset$, $\omega\in \sigma$ for some simplex $\sigma\in\mathscr{P}$. Let the vertices of $\sigma$ be $\{\omega_i\}$, $\psi_F\in \mathbf{R}_{\geqslant 0}^{I_\emptyset}$ means that the inequalities on $\{\omega_i\}$ imply the inequality on $\omega$. Therefore $P_D=P_E$. Consider $g_{\psi_E}$, a $\mathscr{T}$-piecewise affine function. It is strict convex because $D$ is in $C(\mathscr{T})^\circ$. Again multiply $D$ by a positive integer to make $g_{\psi_E}$ integral with respect to $\mathscr{T}$. The corresponding Cartier divisor $E$ is ample for $\mathcal{Y}_\mathscr{T}$. Therefore
\[
\proj \R(\mathcal{Y}_\mathscr{P}, D)=\proj k[S(P_D)]\cong\proj \R(\mathcal{Y}_\mathscr{T}, E)\cong\mathcal{Y}_\mathscr{T}.
\]

It follows that $f_D:\mathcal{Y}_\mathscr{P}\dashrightarrow \mathcal{Y}_\mathscr{T}$ is the rational map defined by $\Sigma_\mathscr{P}\rightarrow \overline{\Sigma}\leftarrow \Sigma_\mathscr{T}$ and $D=f_D^*(E)+F$ for $E$ ample, and $F$ $f_D$-exceptional. That means each $C(\mathscr{T})^\circ$ is contained in one Mori chamber, and different $C(\mathscr{T})^\circ$ are contained in different Mori chambers. Since $C(\mathscr{T})$ and their faces form a complete fan $\Sigma(Q)$ in $\mathbb{L}^*_\mathbf{R}$, the Mori chambers agree with $\{C(\mathscr{T})\}$. As a result, the Mori fan exists and is the same with $\Sigma(Q)$.
\end{proof}

\begin{corollary}
All relative minimal models $\mathcal{Y}\to \overline{\mathcal{Y}}$ are toric models $\mathcal{Y}_\mathscr{P}\to\overline{\mathcal{Y}}$ for some minimal triangulation $\mathscr{P}$.
\end{corollary}

\begin{corollary}
If $\mathcal{Y}_\mathscr{P}$ is a relative minimal model, then
\[
\overline{\effdivisor}(\mathcal{Y}_\mathscr{P})=\effdivisor(\mathcal{Y}_\mathscr{P})=\NS_\mathbf{R}(\mathcal{Y}_\mathscr{P})=\pic(\mathcal{Y}_\mathscr{P})_\mathbf{R}.
\]

and the moving cone is
\[
\mov(\mathcal{Y}_\mathscr{P})=\bigcup_{I\subset \mathscr{T}}C(\mathscr{T}).
\]\label{mov}
\end{corollary}

\begin{corollary}
The union $\bigcup_{I\subset \mathscr{T}}\overline{C}(\mathscr{T})$ is a convex polyhedral cone. 
\end{corollary}
\begin{proof}
This is because the closure of the moving cone is always convex. It is also easy to prove the statement directly.
\end{proof}

We generalize the notion of Mori dream spaces to varieties that are projective over affine varieties. Corollary \ref{mov} implies that the relative minimal models are Mori dream spaces. According to \cite{HK}, the Mori theory here is an instance of the theory of VGIT. In this case, $\mathcal{Y}_\mathscr{P}$ are all GIT quotients by Cox's construction. Define $G=\Hom(\Cl(\mathcal{Y}_\mathscr{P}), \mathbb{G}_m)$. The homomorphism $\mathbf{Z}^{\Sigma(1)}\to \Cl(\mathcal{Y}_\mathscr{P})$ gives a $G$-action on $\mathbf{A}^{\Sigma(1)}$. The problem is how to take the quotient of $\mathbf{A}^{\Sigma(1)}$ by $G$. 

Since the story from this perspective is well presented in \cite{CLS}, we relate our notations with those in \cite{CLS}, for the convenience of the readers. The cone $C_\nu$ is identified with $C(Q)$, and we are in the situation of (\cite{CLS} Proposition 14.3.11). A stability condition is given by a character $\psi\in \widehat{G}_\mathbf{R}=\mathbb{L}^*_\mathbf{R}=C_\beta$. The secondary fan $\Sigma(Q)$ is a complete fan in $C_\beta$. Each maximal cone $C(\mathscr{T})$ is a GIT chamber. For any $\psi\in C(\mathscr{T})^\circ$, we have $\mathbf{A}^{\Sigma(1)}//_\psi G\cong \mathcal{Y}_\mathscr{T}$. Therefore, we say that on the Landau--Ginzburg side, the secondary fan $\Sigma(Q)$ controls the variation of GIT quotients. 


\section{Gluing the Families}\label{glue GKZ families}
We go back to the Fano side and construct the compactification $\mathscr{T}_Q$ of the moduli of toric pairs $(X_Q, \Theta)$.

\begin{definition}
A full dimensional simplex $\sigma\subset \overline{X}_\mathbf{R}$ is called regular if the set $\{(v_i,1)\}$ form a basis of $\mathbb{X}$ for $\{v_i\}$ the set of vertices of $\sigma$. In other words, $\sigma$ is regular if the cone $C(\sigma)$ is a regular polyhedral cone for the lattice $\mathbb{X}$ (\cite{CLS} Definition 1.2.16).
\end{definition}

From now on, assume that $Q$ contains a full dimensional regular simplex. We can always rescale $Q$ so that it contains a standard cube in $\overline{X}_\mathbf{R}$. A standard cube contains a basis for $X$, and thus contains a regular simplex. For the toric variety $X_Q$, rescaling is only changing the polarization. Therefore, this hypothesis is mild. Under this assumption $\mathbb{L}^*$ is free, and the following exact sequence is exact on the right
\begin{equation}
\begin{diagram}
0&\rTo& \mathbb{L}&\rTo& (\mathbf{Z}^I)^*&\rTo^p& \mathbb{X}&\rTo&0.
\end{diagram}\label{exact seq}
\end{equation}

Choose an arbitrary $g$-dimensional regular simplex $\sigma\subset Q$. Define a linear operator $L_\sigma: \mathbf{R}^I \to Aff$.
\begin{definition}
$L_\sigma(\psi)$ is the affine extension to $\overline{X}_\mathbf{R}$ of $\psi|_\sigma$, and then we restrict it to $I$ . 
\end{definition}

\begin{definition}
$\Psi: I\to \mathbb{L}_\mathbf{Q}$ is defined by
\[
\langle\overline{\psi},\Psi(\omega)\rangle=\psi(\omega)-L_\sigma(\psi)(\omega),\  \forall \overline{\psi}\in \mathbb{L}^*_\mathbf{Q},
\]
where $\overline{\psi}$ is the image of $\psi\in \mathbf{Q}^I$. 
\end{definition}

The definition of $\Psi(\omega)$ is independent of the choice of the representative $\psi$. 

\begin{lemma}
The values of $\Psi$ are in $\mathbb{L}$.  
\end{lemma}
\begin{proof}
Since $\sigma$ is regular, $L_\sigma$ maps $\mathbf{Z}^I$ to $Aff(\overline{X},\mathbf{Z})$. If $\overline{\psi}$ is in $\mathbb{L}^*$, $\langle\overline{\psi},\Psi(\omega)\rangle$ is an integer. Hence $\Psi(\omega)\in \mathbb{L}$. 
\end{proof}

Let $\mathscr{T}$ be a coherent triangulation. Denote $\pic(\mathcal{Y}_\mathscr{T})\times \mathbf{Z}^{I_\emptyset}$ by $\mathbb{L}_\mathscr{T}^*$. We choose the integral structure on $\mathbb{L}_\mathbf{R}$ to be the dual $\mathbb{L}_\mathscr{T}$. $\mathbb{L}_\mathscr{T}$ contains $\mathbb{L}$, and depends on $\mathscr{T}$. Define $g_{\Psi,\mathscr{T}}: Q\to \mathbb{L}_\mathbf{R}$ by affinely interpolating $\Psi$ inside each simplex of $\mathscr{T}$. By the mirror interpretation, we have $C(\mathscr{T})=\nef(\mathcal{Y}_\mathscr{T})\times \mathbf{R}_{\geqslant 0}^{I_\emptyset}$, $C(\mathscr{T})^\vee=\overline{\effcurve}(\mathcal{Y}_\mathscr{T})\times (\mathbf{R}_{\geqslant 0}^{I_\emptyset})^*\subset \mathbb{L}_\mathbf{R}$, and $H_\mathscr{T}^{gp}=\pic(\mathcal{Y}_\mathscr{T})^*$. As a result, $S_{C(\mathscr{T})}:=C(\mathscr{T})^\vee\cap\mathbb{L}_\mathscr{T}=H_\mathscr{T}^\sat\times (\mathbf{N}^{I_\emptyset})^*$. If further $I\subset\mathscr{T}$, then $H_\mathscr{T}^\sat=S_{C(\mathscr{T})}$. 

\begin{lemma}
The $\mathscr{T}$-piecewise function $g_{\Psi,\mathscr{T}}$ is equal to the composition of the universal $\mathscr{T}$-piecewise function $\varphi: Q\to H_\mathscr{T}^{\gp}$ in Section~\ref{GKZ standard family} and the map $H_\mathscr{T}^\gp\to \mathbb{L}_\mathscr{T}$. For each codimension-$1$ wall $\rho$, the bending parameter $p_\rho$ is the curve class $[V(C(\rho))]\in \effcurve(\mathcal{Y}_\mathscr{T})$. 
In particular, $g_{\Psi,\mathscr{T}}$ is integral with respect to $\mathbb{L}_\mathscr{T}$.
\end{lemma}
\begin{proof}
For any codimension-$1$ wall $\rho$ that is the intersection of the maximal cells $\sigma_i,\sigma_j\in \mathscr{T}$, we have, for $\psi\in \pic(\mathcal{Y}_\mathscr{T})$, 
\begin{equation}
g_{\Psi,\mathscr{T}}\vert_{\sigma_i}(\psi)-g_{\Psi,\mathscr{T}}\vert_{\sigma_j}(\psi)=\psi\vert_{\sigma_i}-\psi\vert_{\sigma_j}.
\end{equation}

Compare with the bending parameters for $\varphi$ in Equation~\eqref{bending parameters for the standard family in GKZ}, we see the bending parameters for $g_{\Psi_\mathscr{T}}$ are obtained from the bending parameters for $\varphi$ via $H_\mathscr{T}^\gp\to \mathbb{L}_\mathscr{T}$. The interpretation of the bending parameters in terms of curve classes is a standard result in toric geometry (\cite{CLS} Proposition 6.3.8). 
\end{proof}

\begin{remark}
We make a base change from $\mathbb{L}$ to $\mathbb{L}_\mathscr{T}$ because we need the central fiber reduced.
\end{remark}

Construct the polyhedron $Q_{g_{\Psi,\mathscr{T}}}$, the graded ring $R_{g_{\Psi,\mathscr{T}}}=k[S(Q_{g_{\Psi,\mathscr{T}}})]$, and the family $\mathcal{X}_\mathscr{T}:=\proj R_{g_{\Psi,\mathscr{T}}}$ over $U_{C(\mathscr{T})}:=\spec k[S_{C(\mathscr{T})}]$, with the morphism $\pi: \mathcal{X}_\mathscr{T}\to U_{C(\mathscr{T})}$. Take the $\pi$-ample line bundle $\mathcal{L}:=\mathcal{O}(1)$. Define a section of $\mathcal{L}$
\[
\vartheta:=\sum_{\omega\in Q(\mathbf{Z})} X^{(\omega,\Psi(\omega))}.
\]

By definition $(\omega,\Psi(\omega))$ is $C(\mathscr{T})^\vee$-above $(\omega,g_{\Psi,\mathscr{T}}(\omega))$, therefore $(\omega,\Psi(\omega))\in Q_{g_{\Psi,\mathscr{T}}}(\mathbf{Z})$ and $\vartheta$ is a section of $\mathcal{L}$. Take the divisor $\Theta:=(\vartheta)_0$. Moreover, $g_{\Psi,\mathscr{T}}(\omega)=\Psi(\omega)$ if and only if $\omega\in \mathscr{T}\cap I$. By Proposition~\ref{stable toric pairs}, $(\mathcal{X}_\mathscr{T}, \mathcal{L},\Theta,\varrho)$ is a stable toric pair over $U_{C(\mathscr{T})}$. By Cororllary~\ref{GKZ functorial}, $\mathcal{X}_\mathscr{T}$ is the pull--back of the standard family for $\mathscr{T}$. Define the log structures on $\mathcal{X}_\mathscr{T}/U_{C(\mathscr{T})}$ to be the log structures pulled back from the standard family, and denote them by $(\mathcal{X}_\mathscr{T}, P_\mathscr{T})/(U_{C(\mathscr{T})}, M_\mathscr{T})$. 
\begin{lemma}
The family $\pi: (\mathcal{X}_\mathscr{T}, P_\mathscr{T}, \mathcal{L},\vartheta,\varrho)\to (U_{C(\mathscr{T})},M_\mathscr{T})$ is an object in $\mathscr{K}_Q(U_{C(\mathscr{T})})$, where $\mathscr{K}_Q$ is the stack defined in \cite{ols08}. 
\end{lemma}
\begin{proof}
It suffices to check the definition of the stack $\mathscr{K}_Q$ in (\cite{ols08} 3.7.1). In particular, (vii) is satisfied because of (loc. cit. Lemma 3.1.24).
\end{proof}

Define the chart $U=\coprod_{\mathscr{T}} U_{C(\mathscr{T})}$, where the index is over all coherent triangulation $\mathscr{T}$. By the $2$-Yoneda lemma, We can pick a corresponding morphism $F: U\to \mathscr{K}_Q$. The next step is to define the pre-equivalence relation $R\to U\times U=\coprod_{\mathscr{T}_1,\mathscr{T}_2}U_{C(\mathscr{T}_1)}\times U_{C(\mathscr{T}_2)}$. 

First, consider the case $\mathscr{T}_1=\mathscr{T}_2$, and denote it by $\mathscr{T}$. Compare the two lattices $\mathbb{L}^*$ and $\mathbb{L}_\mathscr{T}^*$. The quotient $\mathbb{L}^*/\mathbb{L}_\mathscr{T}^*$ is a finite abelian group. Let $G_\mathscr{T}$ be the kernel of $T_{\mathbb{L}_\mathscr{T}^*}\to T_{\mathbb{L}^*}$, and $m_\mathscr{T}$ be the biggest order of elements in $\mathbb{L}^*/\mathbb{L}^*_\mathscr{T}$. Assume that $m_\mathscr{T}$ is invertible in $k$, then $G_\mathscr{T}$ is isomorphic to the constant group scheme with fiber $\mathbb{L}^*/\mathbb{L}^*_\mathscr{T}$, and is \'{e}tale over $k$. Define $R_\mathscr{T}=U_{C(\mathscr{T})}\times G_\mathscr{T}\to U_{C(\mathscr{T})}\times U_{C(\mathscr{T})}$, where the first factor is the projection and the second factor is the group action. Denote the two projections from $R_\mathscr{T}$ to $U_{C(\mathscr{T})}$ by $s$ and $t$.  Since $G_\mathscr{T}$ is \'{e}tale, $s$ is the first projection $U_{C(\mathscr{T})}\times G_\mathscr{T}$ to $U_{C(\mathscr{T})}$, $s$ is \'{e}tale and surjective. The morphism $t$ is the action $U_{C(\mathscr{T})}\times G_\mathscr{T}\to U_{C(\mathscr{T})}$. It is the composition of the projection and the isomorphism $(x,g)\to (gx, g)$, and is also \'{e}tale and surjective. We get an \'{e}tale pre-equivalence relation $R_\mathscr{T}$. The coarse moduli space of the stack $[U_{C(\mathscr{T})}/R_\mathscr{T}]=[U_{C(\mathscr{T})}/G_\mathscr{T}]$ is the toric variety defined by the cone $C(\mathscr{T})^\vee$ and lattice $\mathbb{L}$.

\begin{proposition}\label{first faithful}
We have a natural injective morphism $R_\mathscr{T}\to U_{C(\mathscr{T})}\times_{\mathscr{K}_Q}U_{C(\mathscr{T})}$. 
\end{proposition}
\begin{proof}
We claim that each element in $G_\mathscr{T}$ induces an isomorphism in $\mathscr{K}_Q(U_{C(\mathscr{T})})$. Let $\zeta$ be an $S$-point of $G_\mathscr{T}$ for some $k$-scheme $S$. The element $\zeta$ is acting on the algebra $R_S:=\mathcal{O}_S[S_{C(\mathscr{T})}]$ by 
\[
g: \mathrm{X}^p\mapsto \mathrm{X}^p(\zeta)\mathrm{X}^p. 
\]

The pull--back $(g^*\mathcal{X}_\mathscr{T},g^*\mathcal{L})$ is $\proj R_S\otimes_{R_S}R_{g_{\Psi,\mathscr{T}}}$, where the map $R_S\to R_S$ is $g$. We use $R_{g_{\Psi,\mathscr{T}}}=R_S[S(Q_{g_{\Psi,\mathscr{T}}})]\subset R_S[\mathbb{X}\times \mathbb{L}_\mathscr{T}]$. Then $(g^*\mathcal{X}_\mathscr{T},g^*\mathcal{L})$ is isomorphic to $(\mathcal{X}_\mathscr{T},\mathcal{L})$ over $U_{C(\mathscr{T})}$ by an isomorphism 
\begin{align*}
F_\zeta: g^*R_{g_{\Psi,\mathscr{T}}}=R_S\otimes_{R_S}R_{g_{\Psi,\mathscr{T}}}&\longrightarrow R_{g_{\Psi,\mathscr{T}}},\\
1\otimes \mathrm{X}^{(n,\alpha,p)}&\longmapsto \mathrm{X}^p(\zeta)\mathrm{X}^{(n,\alpha,p)}.
\end{align*} 

The isomorphism $F_\zeta$ preserves the $\mathbb{X}$-grading. Therefore $G_\mathscr{T}$ is acting on $(\mathcal{X}_\mathscr{T},\mathcal{L})$ commuting with $\pi$ and the action $\varrho$. The action on the log structures are the modifications by $H_\mathscr{T}\to \mathcal{O}_S^*$ and $S(Q)\rtimes H_\mathscr{T}\to \mathcal{O}_S^*$
\begin{align}
p&\mapsto 1,\\
(\alpha,p)&\mapsto \mathrm{X}^{g_{\Psi,\mathscr{T}}(\alpha)}(\zeta) .
\end{align}

The induced log structure is isomorphic to the original log structure,  and thus $G_\mathscr{T}$ preserves the log morphism. The action by $G_\mathscr{T}$ also preserves the section $\vartheta$ because $\Psi(\omega)\in\mathbb{L}$ for all $\omega\in Q(\mathbf{Z})$.

The action of $G_\mathscr{T}$ exchanges with any base change $S\to U_{C(\mathscr{T})}$, and is faithful. 
\end{proof}

Next, consider two different Mori chambers $C(\mathscr{T}_1)$ and $C(\mathscr{T}_2)$. Assume that $\tau:=C(\mathscr{T}_1)\cap C(\mathscr{T}_2)$ is a wall. We want to glue the two associated families together. Regard $\omega\in I$ as a vector in $\mathbb{X}_\mathbf{R}$. According to (\cite{CLS} Chapter 15.3), there are only two cases. 

The first case corresponds to divisorial contraction, and is called the divisorial case. In this case, according to (\cite{CLS} Theorem 15.3.6), one of the triangulation, say $\mathscr{T}_1$, has more vertices than the other. Moreover, there exists $\omega\in I$ such that  $\mathscr{T}_1$ is star subdivision of $\mathscr{T}_2$ at $\omega$. Let $\sigma_0$ be the simplex which contains $\omega$ in $\mathscr{T}_2$. Assume the vertices of $\sigma_0$ are $\{v_0,\ldots,v_n\}$. For any $\sigma_i\in \mathscr{T}_1$ such that $\omega\in \overline{\sigma}_i$, define $\psi_i$ to be the affine function over $\overline{\sigma}_i$ which is $1$ at $\omega$, and $0$ at other vertices. Then 
\begin{align*}
g^{12}&:=g_{\Psi,\mathscr{T}_1}-g_{\Psi,\mathscr{T}_2}\\
&=\left\{
\begin{array}{rl}
0 & \text{if } x\notin \text{Star}(\sigma_0)\\
\psi_i(x)q_\tau &\text{if } x\in\sigma_i
\end{array} \right. ,
\end{align*}

for some $q_\tau\in\mathbb{L}$. 

Write $\omega$ as an affine combination of $\{v_0,\ldots,v_n\}$,
\[
\omega=\sum_{i=0}^n a_i v_i, \quad\text{with }  \sum_{i=0}^g a_i=1.
\]

Compute $q_\tau$,
\begin{align*}
q_\tau&=g^{12}(\omega)\\
&=g_{\Psi,\mathscr{T}_1}(\omega)-g_{\Psi,\mathscr{T}_2}(\omega)\\
&=\Psi(\omega)-\sum_{i=0}^g a_i\Psi(v_i).
\end{align*}

Therefore, for any $\psi'\in\mathbf{R}^I$,
\begin{align*}
q_\tau(\psi')&=(\psi',\Psi(\omega))-\sum_{i=0}^n a_i(\psi',\Psi(v_i))\\
&=\psi'(\omega)-L_\sigma(\psi')(\omega)-\sum_{i=0}^n a_i\psi'(v_i)+\sum_{i=0}^n a_iL_\sigma(\psi')(v_i)\\
&=\psi'(\omega)-\sum_{i=0}^n a_i\psi'(v_i).
\end{align*}

It follows that $-q_\tau\in C(\mathscr{T}_1)^\vee$, and $q_\tau\in C(\mathscr{T}_2)^\vee$. Since $C_\tau^\vee=C(\mathscr{T}_1)^\vee+C(\mathscr{T}_2)^\vee$, $q_\tau\in (S_\tau^*)_\mathbf{Q}$. 

\begin{remark}
Since $\mathscr{T}_1$, $\mathscr{T}_2$ give different integral structures on $\mathbb{L}_\mathbf{R}$, we haven't defined $S_\tau$ yet. However, $(S_\tau^*)_\mathbf{Q}$ is well defined. 
\end{remark}

The second case corresponds to a flip $\mathcal{Y}_{\mathscr{T}_1}\dashrightarrow \mathcal{Y}_{\mathscr{T}_2}$, and is called the flipping case. Since $\mathcal{Y}_{\mathscr{T}_1}$ and $\mathcal{Y}_{\mathscr{T}}$ are isomorphic up to codimension $1$, $\mathscr{T}_1$ and $\mathscr{T}_2$ have the same $I_\emptyset$. Assume $\tau$ corresponds to the paving $\mathscr{P}$. In this case, the wall $\tau$ comes from a wall between $\nef(\mathcal{Y}_{\mathscr{T}_1})$ and $\nef(\mathcal{Y}_{\mathscr{T}_2})$, and thus corresponds to an extremal ray $\mathcal{R}\subset \overline{\effcurve}(\mathcal{Y}_{\mathscr{T}_1})$ or $-\mathcal{R}\subset \overline{\effcurve}(\mathcal{Y}_{\mathscr{T}_2})$. This curve class defines two sets
\[
J_-:=\{v\in I: D_v\cdot \mathcal{R}<0\}, \quad J_+:=\{v\in I: D_v\cdot \mathcal{R}>0\}. 
\]

Also for any $J\subset I\backslash I_\emptyset$, set 
\[
\sigma_J:= \cone(\{v: v\in J\})\subset \mathbb{X}_\mathbf{R}. 
\]

Here are the facts we need from (\cite{CLS} Theorem 15.3.13). 
\begin{itemize}
\item[a)] Both $J_+$ and $J_-$ have at least $2$ elements.
\item[b)] Vectors in $J_-$ and $J_+$  form an oriented circuit. That means there is one linear relationship
\[
\sum_{i\in J_-}b_i v_i+\sum_{i\in J_+}b_i v_i=0,
\]

where $b_i>0$, if $i\in J_+$; $b_i<0$, if $i\in J_-$. And every proper subset is linearly independent. We normalize it so that $\sum_{i\in J_+} b_i=-\sum_{i\in J_-}b_i=1$. 
\item[c)] $\sigma_{J_-}\in \Sigma_{\mathscr{T}_1}$, $\sigma_{J_+}\in \Sigma_{\mathscr{T}_2}$, and $\sigma_{J_-\cup J_+}\in \Sigma_{\mathscr{P}}$. 
\item[d)] All the non-simplicial cones of $\Sigma_{\mathscr{P}}$ are contained in Star$(\sigma_{J_-\cup J_+})$. And
\[
\Sigma_{\mathscr{T}_1}\backslash \text{Star}(\sigma_{J_-})=\Sigma_{\mathscr{P}}\backslash \text{Star}(\sigma_{J_-\cup J_+})=\Sigma_{\mathscr{T}_2}\backslash \text{Star}(\sigma_{J_+}).
\]
\item[e)] For any maximal non-simplicial cell $\sigma_\alpha\in \Sigma_\mathscr{P}$, there is a set $J_\alpha\subset I\backslash I_\emptyset$, such that $J_\alpha\cup J_-\cup J_+$ is the set of vertices of $\sigma_\alpha$, and $|J_\alpha\cup J_-\cup J_+|=g+2$. Denote $\Sigma_{\mathscr{T}_1}|_{\sigma_\alpha}$ by $\Sigma_-$, and $\Sigma_{\mathscr{T}_2}|_{\sigma_\alpha}$ by $\Sigma_+$, then we have 
\[
\Sigma_-=\{\sigma_J: J\subset J_\alpha\cup J_+\cup J_-, J_+\nsubseteq J\}, \quad \Sigma_+=\{\sigma_J: J\subset J_\alpha\cup J_+\cup J_-, J_-\nsubseteq J\}. 
\]
\end{itemize}

Set
\[
\omega=\sum_{i\in J_+}b_i v_i=\sum_{i\in J_-}-b_i v_i.
\]

For each maximal non-simplicial cell $\sigma_\alpha$, let $J^\alpha=J_\alpha\cup J_-\cup J_+$. For any pair $i\in J_-, j\in J_+$, Define $\sigma_{ij}=\sigma_{J^\alpha\backslash\{i\}}\cap \sigma_{J^\alpha\backslash\{j\}}$. Since vectors in $J_+$ and $J_-$ form a circuit
\[
\sigma_{ij}=\cone(\{\omega, v: v\in J^\alpha\backslash\{i,j\}\})
\]

and is a simplex. 

For any such $\sigma_{ij}$, define $\psi_{ij}$ to be the affine function over $\overline{\sigma_{ij}}$ that is $1$ on $\omega$ and is $0$ on other vertices. Consider
\begin{align*}
g^{12}&:=g_{\Psi,\mathscr{T}_1}-g_{\Psi,\mathscr{T}_2}\\
&=\left\{
\begin{array}{rl}
0 & \text{if } x\in |\Sigma_\mathscr{P}\backslash \text{Star}(\sigma_{J_-\cup J_+})|\\
\psi_{ij}(x)q_\tau &\text{if } x\in\sigma_{ij}
\end{array} \right. ,
\end{align*}

for some $q_\tau\in\mathbb{L}$.

Compute $q_\tau$,
\begin{align*}
q_\tau&=g^{12}(\omega)\\
&=g_{\Psi,\mathscr{T}_1}(\omega)-g_{\Psi,\mathscr{T}_2}(\omega)\\
&=\sum_{i\in J_-}(-b_i)\Psi(v_i)-\sum_{i\in J_+}b_i\Psi(v_i)
\end{align*}

Therefore, for any $\psi'\in \mathbf{R}^I$, 
\begin{equation}
q_\tau(\psi')=\sum_{i\in J_-}(-b_i)\psi'(v_i)-\sum_{i\in J_+}b_i\psi'(v_i)
\end{equation}

Since $J_+$ has at least two elements, pick $v_k, v_l$ from $J_+$. $\sigma_k:=\sigma_{J\backslash\{k\}}$ and $\sigma_l:=\sigma_{J\backslash\{l\}}$ are two maximal cones in $\Sigma_-$. $\varsigma:=\sigma_{J\backslash\{k,l\}}$ is a wall between them. By (\cite{CLS} Proposition 6.4.4.)
\[
V(C(\varsigma))(\psi')=\frac{\mult(\varsigma)}{\mult(\sigma_k)(-b_k)}\Big(\sum_{i\in J_-}(-b_i)\psi'(v_i)-\sum_{i\in J_+}b_i\psi'(v_i)\Big)
\]

Therefore 
\[
q_\tau=\frac{\mult(\sigma_k)(-b_k)}{\mult(\varsigma)}[V(C(\varsigma))]=\frac{\mult(\sigma_k)(-b_k)}{\mult(\varsigma)}p_\varsigma.
\]

Since $[V(C(\varsigma))]$ is a curve class in $\mathcal{R}$, $q_\tau$ is a curve class in $\mathcal{R}$. It follows that $q_\tau\in C(\mathscr{T}_1)^\vee$, and $-q_\tau\in C(\mathscr{T}_2)^\vee$. Again $q_\tau\in (S_\tau^*)_\mathbf{Q}$. 

\begin{proposition}\label{adjacent}
If $\mathscr{T}_1$ and $\mathscr{T}_2$ are coherent triangulations such that $C(\mathscr{T}_1)$ and $C(\mathscr{T}_2)$ are of maximal dimension, and $\tau=C(\mathscr{T}_1)\cap C(\mathscr{T}_2)$ is a codimension-$1$ wall. Then $g^{12}=g_{\Psi,\mathscr{T}_1}-g_{\Psi,\mathscr{T}_2}$ takes values in $(S^*_\tau)_{\mathbf{Q}}$.
\end{proposition}

\begin{corollary}\label{invertible}
If $\mathscr{T}_1$ and $\mathscr{T}_2$ are coherent triangulations such that $C(\mathscr{T}_1)$ and $C(\mathscr{T}_2)$ are of maximal dimension. Let $\tau=C(\mathscr{T}_1)\cap C(\mathscr{T}_2)$ be the common face. Then $g^{12}=g_{\Psi,\mathscr{T}_1}-g_{\Psi,\mathscr{T}_2}$ takes values in $(S^*_\tau)_{\mathbf{Q}}$.
\end{corollary}

\begin{proof}
First we claim that $C(\mathscr{T}_1)$ and $C(\mathscr{T}_2)$ can be connected by a series of adjacent maximal cones $\{C_i\}_{0\leqslant i\leqslant l}$ such that $C_0=C(\mathscr{T}_1)$, $C_l=C(\mathscr{T}_2)$, and $\tau\subset C_i$ for all $i$. For the proof, look at the secondary polytope $P(Q)\subset\mathbb{L}_\mathbf{R}$. $\tau$ corresponds to a face $F_\tau$ of $X$ which is itself a polytope. $C(\mathscr{T}_1)$ and $C(\mathscr{T}_2)$ correspond to two vertices $v_1$ and $v_2$ of $F_\tau$. And they can be connected by edges of $F_\tau$. The vertices on these edges correspond to $C_i$ we are seeking. This proves the claim.

Define $g^{i,i-1}$ for $C_{i-1}$ and $C_i$ as in the proposition \ref{adjacent}. Then $g^{12}=g_{\Psi,\mathscr{T}_1}-g_{\Psi,\mathscr{T}_2}=\sum_{i=1}^lg^{i,i-1}$, and $(\sum_{i=0}^l S_{C_i})_\mathbf{Q}\subset (S_\tau)_\mathbf{Q}=(S_{C_0}+S_{C_l})_\mathbf{Q} \subset (\sum_{i=0}^l S_{C_i})_\mathbf{Q}$. It follows that $g^{12}$ takes values in $(S_\tau^*)_\mathbf{Q}$. 
\end{proof}

Let $\mathscr{T}_1$ and $\mathscr{T}_2$ be two coherent triangulations such that $C(\mathscr{T}_1)$ and $C(\mathscr{T}_2)$ are maximal cones. Let $\tau=C(\mathscr{T}_1)\cap C(\mathscr{T}_2)$. Define $\mathbb{L}_\tau$ to be the lattice generated by $S_{C(\mathscr{T}_1)}^{\gp}$ and $S_{C(\mathscr{T}_2)}^{\gp}$. Since $S_{C(\mathscr{T}_1)}^{\gp}$ and $S_{C(\mathscr{T}_2)}^{\gp}$ are commensurable, $\
\mathbb{L}_\tau$ is commensurable to both of them. Define $S_\tau=(C(\mathscr{T}_1)^\vee+C(\mathscr{T}_2)^\vee)\cap \mathbb{L}_\tau$.  Let $U_\tau:=\spec \mathbf{Z}[S_\tau]$. The inclusions of monoids  $S_{C(\mathscr{T}_i)}\to \tau^\vee\cap S_{C(\mathscr{T}_i)}^{\gp}\to S_\tau$ define morphisms $p_{\tau,\mathscr{T}_i}: U_\tau\to U_{C(\mathscr{T}_i)}$. 

\begin{lemma}\label{lattices agree for GKZ}
The morphisms $p_{\tau,\mathscr{T}_i}: U_\tau \to U_{C(\mathscr{T}_i)}$ are both \'{e}tale for $i=1,2$. 
\end{lemma}
\begin{proof}
Denote the lattice $\mathbb{L}_{\mathscr{T}_i}$ by $\mathbb{L}_i$. This is the lattice used to define $S_{C(\mathscr{T}_i)}$. We only need to show that the morphism $U_\tau\to \spec k[ \tau^\vee\cap \mathbb{L}_i]$ is \'{e}tale. The cone $\tau$ is associated to the paving $\mathscr{P}$. Let $N_\tau$ be the lattice $PA(\mathscr{P},\mathbf{Z})/Aff=\pic(\mathcal{Y}_\mathscr{P})$. We claim that $N_\tau=N_{\tau,\mathbf{R}}\cap \mathbb{L}^*_i$ for $i=1,2$. Notice that $\mathscr{P}$ is coarser than both $\mathscr{T}_i$. Assume that $\{\sigma_{jk}\}$ is a collection of top-dimensional cells in $\mathscr{T}_i$, and for each $k$, the union of the closure of $\sigma_{jk}$ is the closure of a top-dimensional cell $\sigma_k$ in $\mathscr{P}$. Then a piecewise affine function $\psi$ in $N_\tau$ just means it is affine on each $\sigma_k$ and are integral on each of top-dimensional cells. Therefore it is integral on each $\sigma_{jk}$, and it is in $N_{\tau,\mathbf{R}}\cap \mathbb{L}_i^*$. On the other hand, if $\psi\in N_{\tau,\mathbf{R}}\cap \mathbb{L}_i^*$, then it is affine on each $\sigma_k$ and integral on each $\sigma_{jk}$. Since $\sigma_{jk}$ is of top dimension, $\psi$ is integral on $\sigma_{k}$ by our definition of integrality. The claim is proved. Let  $I$ be the subset of vertices that are not in either of $\mathscr{T}_i$, i.e. $I:=I_\emptyset^1\cap I_\emptyset^2$. Recall the cone $\tau=C(\mathscr{P},\mathbf{R})\times \mathbf{R}^{I}_{\geqslant 0}\times \{0\}$.  As a result, we have the exact sequence,
\[
\begin{diagram}
0&\rTo&(S^*_\tau)_\mathbf{Q}\cap S_{C(\mathscr{T}_i)}^{\gp}&\rTo&S_{C(\mathscr{T}_i)}^{\gp}&\rTo& H_\mathscr{P}^{\gp}\times (\mathbf{Z}^I)^*&\rTo&0.
\end{diagram}
\]

The image of the quotient $H_\mathscr{P}^{\gp}\times (\mathbf{Z}^I)^*$ is independent of $i$. Intersecting $\tau^\vee$, we have 
\begin{equation}\label{independence}
\begin{diagram}
0&\rTo&(S^*_\tau)_\mathbf{Q}\cap \mathbb{L}_i&\rTo&\mathbb{L}_i\cap \tau^\vee &\rTo&H_\mathscr{P}^\sat\times (\mathbf{N}^I)^*&\rTo&0\\
&&\dTo^{f^\flat}&&\dTo&&\dTo^{=}&&\\
0&\rTo&(S^*_\tau)_\mathbf{Q}\cap \mathbb{L}_\tau&\rTo&S_\tau &\rTo&H_\mathscr{P}^\sat\times(\mathbf{N}^I)^*&\rTo&0
\end{diagram},
\end{equation}

where $f^\flat$ is an injection with image of finite index. By (\cite{ogus} Chap.\Rmnum{1} Proposition 1.1.4 part 2), the first square is a push--out diagram. Therefore, we get a pul--back diagram
\begin{diagram}
U_{\tau}&\rTo^{g'}& \spec k[(S_\tau^*)_\mathbf{Q}\cap \mathbb{L}_i]\\
\dTo^{p_{\tau,\mathscr{T}_i}}&&\dTo_{f}\\
\spec k[\tau^\vee\cap \mathbb{L}_i]&\rTo^{g} &\spec k[(S_\tau^*)_\mathbf{Q}\cap\mathbb{L}_\tau],
\end{diagram}

where $g$ and $g'$ are fibrations with fibers $\spec k[H_\mathscr{P}^\sat\times(\mathbf{N}^I)^*]$. Since $f$ is \'{e}tale, $p_{\tau,\mathscr{T}_i}$ is \'{e}tale. 
\end{proof}

\begin{corollary}\label{GKZ log}
Let $F_i$ be the face of $H_{\mathscr{T}_i}^\sat\cap \tau^\perp$, and $H_{\mathscr{T}_i,F_i}^\sat$ be the localization with respect to $F_i$. The following diagram commutes .
\begin{diagram}
C(\mathscr{T}_i)^\vee\cap \mathbb{L}_i&\rTo &\tau^\vee \cap \mathbb{L}_\tau\\
\uTo &&\uTo\\
H_{\mathscr{T}_i}^\sat&\rTo & H_{\mathscr{P}}^\sat\oplus F_i^{\gp},
\end{diagram}

where the bottom line is the localization. 
\end{corollary}

\begin{proof}
It is implied by the diagram \eqref{independence} in the above proof. We use the notations in the above proof. Let $I_i=I_\emptyset^i\backslash I$. Recall $C(\mathscr{T}_i)^\vee\cap \mathbb{L}_i=H_{\mathscr{T}_i}^\sat\times (\mathbf{N}^{I_\emptyset^i})^*$. The localization of $C(\mathscr{T}_i)^\vee\cap \mathbb{L}_i$ with respect to the face $\tau^\perp\cap \mathbb{L}_i$ is $H_{\mathscr{T}_i,F_i}^\sat\times (\mathbf{Z}^{I_i})^*\times (\mathbf{N}^I)^*$. Therefore, the top exact sequence in \eqref{independence} shows that $H_\mathscr{P}^\sat$ is the quotient of $H_{\mathscr{T}_i,F_i}^\sat$ by the face $F_i^{\gp}$. Since all the monoids are toric, we can choose splittings of the exact sequences. It is the commutative diagram we want. 
\end{proof}

Denote the group scheme in the \'{e}tale pre-equivalence relation for $U_{C(\mathscr{T}_i)}$ by $G_i$, and the group action by $\rho_i: G_i\times U_{C(\mathscr{T}_i)}\to U_{C(\mathscr{T}_i)}$ . Define $R_{ij}:=G_i\times U_\tau\times G_j$ for $i\neq j$. Define $G_i\times U_\tau\to U_{C(\mathscr{T}_i)}$ by $\rho_i\circ(\Id, p_{\tau,\mathscr{T}_i})$. Then composed with the projection $G_i\times U_\tau\times G_j\to G_i\times U_\tau$, we get a morphism $s: R_{ij}\to U_{C(\mathscr{T}_i)}$. Define $t: R_{ij}\to U_{C(\mathscr{T}_j)}$ similarly. By Lemma~\ref{lattices agree for GKZ}, both $s$ and $t$ are \'{e}tale. 

Denote the toric monoid $S_{C(\mathscr{T}_i)}$ by $P_i$ for $i=1,2$. We have defined $k[P_i ]$-algebras $R_i:=R_{g_{\Psi,\mathscr{T}_i}}$. The pull--back along $p_{\tau,\mathscr{T}_i}$ is $R_i\vert_{U_\tau}=k[S_\tau]\otimes_{k[P_i]} R_i$ is
\[
k[S_\tau]\otimes_{k[P_i]}R_i\cong k[S(Q)\rtimes_i S_\tau].
\]

The addition of $S(Q)\rtimes_i S_\tau$ is defined as
\[
(\alpha,p)+(\beta,q)=(\alpha+\beta, p+q+ng_{\Psi,\mathscr{T}_i}(\alpha)+mg_{\Psi,\mathscr{T}_i}(\beta)-(n+m)g_{\Psi,\mathscr{T}_i}(\gamma)).
\]
We use the subscript $\rtimes_i$ to distinguish different additions on the same underlying set. 

Consider $g^{12}=g_{\Psi,\mathscr{T}_1}-g_{\Psi,\mathscr{T}_2}$. Since $g_{\Psi,\mathscr{T}_i}$ are integral with respect to each integral structure, and $\mathbb{L}_\tau$ is a refinement of both the integral structure, $g^{12}$ is integral for $\mathbb{L}_\tau$. Therefore we can define a map
\begin{align*}
S(Q)\rtimes_1 S_\tau&\longrightarrow S(Q)\rtimes_2 S_\tau\\
(\alpha,p)&\longmapsto (\alpha, p+\deg(\alpha)g^{12}(\alpha)).
\end{align*}

It is a morphism between monoids and preserves the action of $S_\tau$. Moreover, by Corollary~\ref{invertible}, this is an isomorphism. Therefore, it induces an isomorphism between graded $k[S_\tau]$-algebras 
\[
\varphi_{\mathscr{T}_1\mathscr{T}_2}: k[S(Q)\rtimes_1S_\tau]\to k[S(Q)\rtimes_2S_\tau].
\]
 
Denote the pull back of family $\mathcal{X}_{\mathscr{T}_i}$ to $U_\tau$ by $\mathcal{X}_i$. It is defined as $\proj k[S(Q)\rtimes_iS_\tau]$. The isomorphism $\varphi_{\mathscr{T}_1\mathscr{T}_2}$ induces an isomorphism between $\mathcal{X}_i$'s.

The section $\vartheta$ for $U_{C(\mathscr{T})}$ is defined by $\sum_{\omega\in I}\mathrm{X}^{\Psi(\omega)-g_{\Psi,\mathscr{T}}(\omega)}\vartheta_\omega$, and is preserved by $\varphi_{\mathscr{T}_1\mathscr{T}_2}$. Furthermore, the $\mathbb{X}$-grading is preserved. In other words, the line bundles $\mathcal{L}_1$, $\mathcal{L}_2$, the sections $\Theta_1$, $\Theta_2$, the $\mathbb{T}$-action $\varrho_1$, $\varrho_2$ are all compatible under the isomorphism $\varphi_{\mathscr{T}_1\mathscr{T}_2}$. The only thing left is the log structure. 

\begin{proposition} \label{GKZlog}
The log structures on $U_{C(\mathscr{T}_i)}$ agree on $U_\tau$, and has a chart $H_\mathscr{P}$. The log structures induced from $P_i$ on $\mathcal{X}_i$ agree by the isomorphism $\varphi_{\mathscr{T}_1\mathscr{T}_2}$, and has a chart $S(Q)\rtimes H_\mathscr{P}$. For any geometric point $\overline{s}$ in the interior of the closed orbit corresponding to $\tau$, the above charts are good at $\overline{s}$.\end{proposition}
\begin{proof}
By Corollary \ref{GKZ log}, the map $H_{\mathscr{T}_i}\to H_{\mathscr{T}_i}^\sat\to H_\mathscr{P}^\sat\oplus F_i^{\gp}$ is a chart for the log structure $M_{U_{C(\mathscr{T}_i)}}$ on $U_\tau$.  Choose a geometric point $\overline{s}$ in the interior of the closed orbit corresponding to $\tau$. It induces a geometric point of $U_{C(\mathscr{T}_i)}$, still denoted by $\overline{s}$, by the \'{e}tale map $U_\tau\to U_{C(\mathscr{T}_i)}$. Let the residue field of $\overline{s}$ be $k(\overline{s})$. Therefore, we get a map $H_{\mathscr{T}_i}^\sat\to k(\overline{s})$ such that $F_i$ is the inverse image of $k(\overline{s})^*$. The quotient $H_{\mathscr{T}_i}^\sat/F_i=H_\mathscr{P}^\sat$ by Corollary \ref{GKZ log}.  It induces a morphism $f: H_{\mathscr{T}_i}\to k(\overline{s})$. The face $F_i':=f^{-1}(k(\overline{s})^*)$ is the preimage of $F_i$ in $H_{\mathscr{T}_i}\to H_{\mathscr{T}_i}^\sat$. We have the commutative diagram
\begin{diagram}
0&\rTo&F'_i&\rTo&H_{\mathscr{T}_i}&\rTo&H_\mathscr{P}&\rTo&0\\
&&\dTo &&\dTo &&\dTo &&\\
0&\rTo &F_i&\rTo &H_{\mathscr{T}_i}^\sat&\rTo &H_\mathscr{P}^\sat&\rTo &0
\end{diagram}

By Corollary \ref{GKZ log}, the bottom sequence is exact. We claim that the top sequence is also exact. Since $H_\mathscr{P}$ is integral, $H_\mathscr{P}\to H_{\mathscr{P}}^\sat$ is injective. It follows that $F'_i$ is mapped to $0$ in $H_\mathscr{P}$. By the definition of $F_i'$, the left square is cartesian. Therefore, the top sequence is exact except possibly at $H_\mathscr{P}$. By (\cite{ols08} Corollary 3.1.22), the quotient $H_{\mathscr{T}_i}/F_i'$ is actually $H_{\mathscr{P}}$. It follows that $H_\mathscr{P}\to H_\mathscr{P}^\sat\oplus F_i^{\gp}\to k[S_\tau]$ is a fine chart. This log structure is independent of the choice of splittings and the choice of $i$.

For the same reason, we can use the chart $S(Q)\rtimes H_\mathscr{P}\to k[S(Q)\rtimes_i S_\tau]$ for the log structure from $(\mathcal{X}_i, P_i)$. We can embed them into charts $(S(Q)\rtimes H_\mathscr{P})\oplus S_\tau^*\to k[S(Q)\rtimes_iS_\tau]$. The isomorphism $\varphi_{\mathscr{T}_1\mathscr{T}_2}$ induces an isomorphism between the pre-log structures $(S(Q)\rtimes H_\mathscr{P})\oplus S^*_\tau\to k[S(Q)\rtimes_i S_\tau]$.

For any geometric point $\overline{s}$ in the interior of the closed orbit corresponding to $\tau$, $H_\mathscr{P}$ is sharp, and $H_\mathscr{P}\backslash\{0\}$ is mapped to $0$ in $k(\overline{s})$. 
\end{proof}
 
 It follows that $\varphi_{\mathscr{T}_1\mathscr{T}_2}$ is a morphism in $\mathscr{K}_Q(U_\tau)$. Use the composition rule in the stack $\mathscr{K}_Q$,  

\begin{proposition}\label{second faithful}
We have an injective morphism $R_{12}\to U_{C(\mathscr{T}_1)}\times _{\mathscr{K}_Q}U_{C(\mathscr{T}_2)}$ as a morphism between sheaves over the big \'{e}tale site.
 \end{proposition}

Let's also denote $R_\mathscr{T}$ by $R_{\mathscr{T}\mathscr{T}}$ and define $R:=\coprod_{\mathscr{T}_1,\mathscr{T}_2} R_{\mathscr{T}_1\mathscr{T}_2}$. We have two \'{e}tale surjective morphisms $s,t: R\to U$ and the morphism $(s,t): R\to U\times U$ is finite. Let $R\times _{s,U,t}R$ be $X_2$. Define a morphism $\mu: X_2\to R$ as follows. For convenience, we also use the following notations. $R_{11}$ or $R_{22}$ means $R_{\mathscr{T}}$, $R_{12}$ or $R_{23}$ means $R_{\mathscr{T}_1\mathscr{T}_2}$ with $\mathscr{T}_1\neq\mathscr{T}_2$. $X_2$ is a disjoint union of the following different types.
\[
R_{11}\times_UR_{11}, R_{21}\times_UR_{11}, R_{11}\times_U R_{12}, R_{12}\times_UR_{23}. 
\]

For $R_{11}\times_UR_{11}$, denote the corresponding \'{e}tale chart by $U_1$ and the finite group scheme by $G_1$. $R_{11}=U_1\times G_1$. Define $\mu: R_{11}\times_UR_{11}\to R_{11}$ by,
\[
R_{11}\times_UR_{11}=(U_1\times G_1)_s\times_{U_1,t}U_1\times G_1\cong G_1\times (G_1\times U_1)\stackrel{m}{\to} G_1\times U_1\cong R_{11}. 
\]

The isomorphism "$\cong$" is due to the fact that $s$ is the projection. $m$ is the multiplication for the group scheme $G_1$. 

Let $\tau$ be the intersection of $\mathscr{T}_1$ and $\mathscr{T}_2$. Similarly for $R_{21}\times_UR_{11}$ and $R_{11}\times_UR_{12}$, the morphism $\mu$ is defined by 
\[
R_{11}\times_UR_{12}=(U_1\times G_1)_s\times_{U_1,t}(G_1\times U_\tau\times G_2)\cong G_1\times G_1\times U_\tau\times G_2\stackrel{m}{\to} G_1\times U_\tau\times G_2=R_{12},
\]

and 
\begin{align*}
R_{21}\times_UR_{11}&=(G_2\times U_\tau\times G_1)_s\times_{U_1,t}(U_1\times G_1)\\
&=G_2\times (U_\tau\times G_1)_s\times_{U_1,t}(U_1\times G_1)\\
&\cong G_2\times (U_\tau\times_{U_1,t}(U_1\times G_1))_s\times_{U_1,t}(U_1\times G_1)\\
&\cong G_2\times U_\tau\times_{U_1,t}((U_1\times G_1)_s\times_{U_1,t}(U_1\times G_1))\\
&\to G_2\times U_\tau\times_{U_1,t}(U_1\times G_1)\\
&\cong G_2\times U_\tau\times G_1=R_{21}.
\end{align*}

Here we used $\mu$ for $R_{11}$ and the automorphism $(g,x)\mapsto (g,gx)$ of $G_1\times U_1$ to replace $s$ by $t$. 

For $R_{12}\times_UR_{23}$, we can play the trick to compose the morphisms from $G_i$. So it surffices to define the morphism $U_\tau\times_{U_2} U_\upsilon\to U_\rho$, where $\tau$ is the intersection of $C(\mathscr{T}_1)$ and $C(\mathscr{T}_2)$, $\upsilon$ is the intersection of $C(\mathscr{T}_2)$ and $C(\mathscr{T}_3)$, and $\rho$ is the intersection of $C(\mathscr{T}_1)$ and $C(\mathscr{T}_3)$. The point is the intersection $\varrho=\overline{\tau}\cap \overline{\upsilon}$ is equal to the intersection $C(\mathscr{T}_1)\cap C(\mathscr{T}_2)\cap C(\mathscr{T}_3)$. The affine toric variety $U_\tau\times_{U_2}U_\upsilon$ is defined by the cone $\varrho^\vee$ with the lattice $\mathbb{L}_\tau+\mathbb{L}_\upsilon=\mathbb{L}_1+\mathbb{L}_2+\mathbb{L}_3$. The cone $\rho^\vee$ is contained in the cone $\varrho^\vee$ and the lattice $\mathbb{L}_\tau+\mathbb{L}_\upsilon$ is finer than the lattice $\mathbb{L}_\rho$. Therefore there is a natural inclusion between the monoids which induces a morphism $U_\tau\times_{U_2} U_\upsilon\to U_\rho$. This gives the desired morphism $\mu$. 

In summary, we have defined a morphism $\mu: X_2\to R$. Define $\epsilon: U\to R$ by the identity section $U_{C(\mathscr{T})}\to U_{C(\mathscr{T})}\times G=R_\mathscr{T}$. Let $\mathcal{C}$ be a category in which finite fiber products always exist. Recall the definition of a groupoid in $\mathcal{C}$ from (\cite{FC} Page 20).
\begin{definition}
A groupoid in $\mathcal{C}$ consists of the following data:
\begin{enumerate}
\item Two objects $X_0, X_1$ in $\mathcal{C}$ and two morphisms $s,t: X_1\to X_0$. 
\item a morphism $\mu: X_2:=X_s\times_{X_0,t}X_1\to X_1$. 
\end{enumerate}
Let $p_1,p_2$ be the first and second projections from $X_2$ to $X_1$. These data should satisfy the following conditions:
\begin{itemize}
\item[a)] The diagrams below are Cartesian.
\begin{diagram}
X_2 &\rTo^\mu &X_1&&&X_2 &\rTo^\mu &X_1&&&X_2 &\rTo^{p_2} &X_1\\
\dTo^{p_1}& &\dTo_t&&&\dTo^{p_2}& &\dTo_s&&&\dTo^{p_1}& &\dTo_t\\
X_1&\rTo^t &X_0&&&X_1&\rTo^s&X_0&&&X_1&\rTo^s &X_0.
\end{diagram}
\item[b)]
The following diagram commutes.
\begin{diagram}
X_1\times_{s,t}X_1\times_{s,t}X_1&\rTo^{\mu\times \Id}& X_1\times_{s,t}X_1\\
\dTo^{\Id\times \mu} &&\dTo_{\mu}\\
X_1\times_{s,t}X_1&\rTo^\mu &X_1.
\end{diagram}
\item[c)] There exists a morphism $\epsilon: X_0\to X_1$ such that $s\circ\epsilon=t\circ\epsilon=\Id$. 
\end{itemize}
\end{definition}

\begin{proposition}
The data $(U, R, s,t,\mu,\epsilon)$ is a groupoid in $(\Sch/k)$. 
\end{proposition}
\begin{proof}
First, the third diagram in a) is automatically Cartesian by the definition of $X_2$. For the rest conditions, it is equivalent to show that for any $S\in(\Sch/k)$, the data $(U(S), R(S), s,t,\mu,\epsilon)$ is a groupoid in sets. Therefore we check a), b), and c) for $S$-points. Divide into four different cases. For $R_{11}\times_UR_{11}, R_{21}\times_UR_{11}$ and $R_{11}\times_U R_{12}$, it reduces to the group actions on sets. For $R_{12}\times_UR_{23}$ case, recall that $U_\tau\times_{U_2}U_\upsilon$ is an affine toric variety defined by the monoid $\varrho^\vee\cap (\mathbb{L}_1+\mathbb{L}_2+\mathbb{L}_3)$. It follows that the following diagrams are Cartesian. 
\begin{diagram}
U_\tau\times_{U_2}U_\upsilon &\rTo^\mu &U_\rho &&&U_\tau\times_{U_2}U_\upsilon &\rTo^\mu &U_\rho\\
\dTo^{p_1}& &\dTo_{t}&&&\dTo^{p_2}& &\dTo_s\\
U_\tau&\rTo^t&U_1&&&U_\upsilon&\rTo^s&U_3
\end{diagram}

Then we can use $S$-points and check the diagrams for sets. 
\end{proof}

By (\cite{LMB} (3.4.3)), one can associate a stack $[U/R]$ with respect to the \'{e}tale topology on $(\Sch/k)$. Denote $[U/R]$ by $\mathscr{T}_Q$. 

\begin{theorem}
The stack $\mathscr{T}_Q$ is a proper Deligne--Mumford stack with finite diagonal. It admits a coarse moduli space $\mathcal{T}_Q$. 
\end{theorem}
\begin{proof}
Since $s$ and $t$ are both \'{e}tale, and $(s,t): R\to U\times U$ is finite, by (\cite{LMB} Proposition (4.3.1)), $\mathscr{T}_Q$ is a Deligne--Mumford stack, and the canonical morphism $p: U\to \mathscr{T}_Q$ is an \'{e}tale presentation. Since being finite is a property stable under base change and local in the \'{e}tale topology on target (loc. cit. (3.10)), consider the cartesian diagram
\begin{diagram}
U\times_{\mathscr{T}_Q}U&\rTo^i & U\times U\\
\dTo &&\dTo\\
\mathscr{T}_Q&\rTo^{\Delta} &\mathscr{T}_Q\times\mathscr{T}_Q
\end{diagram}

By (loc. cit. (3.4.3)), the fiber product $U\times_{\mathscr{T}_Q}U$ is isomorphic with $R$ and the morphism $i$ is the finite morphism $R\to U\times U$. This implies that the diagonal $\Delta$ is finite. By (\cite{ols08} Theorem 1.4.2), $\mathscr{T}_Q$ admits a coarse moduli space $\mathcal{T}_Q$. Since $\Delta$ is proper, $\mathscr{T}_Q$ is separated. Since $U$ is quasi-compact, $\mathscr{T}_Q$ is quasi-compact. By the paragraph after 1.3 in \cite{ols04}, we can use the valuative criterion of properness for discrete valuation rings . Furthermore, we can assume that the DVR $(R,K)$ is complete with respect to the maximal ideal $\mathfrak{m}$, and the generic point $\eta$ is in the open substack that corresponds to geometrically irreducible fiber $X_\eta$. After an \'{e}tale base change, we can assume that $X_\eta$ admits a $K$-point in the open torus orbit. The proof essentially follows from the proof of (\cite{Alex02} Theorem 2.8.1) or (\cite{ols08} Lemma 3.7.8). Choose a uniformizer $s$. Let $\mathscr{P}$ be the decomposition, and $\psi: Q(\mathbf{Z})\to \mathbf{Z}$ be the integral valued function obtained in the above proofs. Then, after a finite base change if necessary, $\psi\in C(\mathscr{P})\cap \pic\times\mathbf{Z}^{I_\emptyset}$. This defines a morphism $\spec R\to U_{C(\mathscr{P})}$. The pull back of the standard family extends $X_\eta$ to the whole $\spec R$. 
\end{proof}

\begin{lemma}\label{cocycle condition}
Denote the composition of morphism in the stack $\mathscr{K}_Q$ also by $\mu'$. The following diagram commutes.
\begin{diagram}
R_{12,s}\times_{U_2,t}R_{23}&\rTo & (U_{C(\mathscr{T}_1)}\times_{\mathscr{K}_Q}U_{C(\mathscr{T}_2)})\times_U(U_{C(\mathscr{T}_2)}\times_{\mathscr{K}_Q}U_{C(\mathscr{T}_3)})\\
\dTo^{\mu} &&\dTo_{\mu'}\\
R_{13} &\rTo &U_{C(\mathscr{T}_1)}\times_{\mathscr{K}_Q}U_{C(\mathscr{T}_3)}
\end{diagram}
\end{lemma}

\begin{proof}
The only nontrivial case to check is $1\neq 2$ and $2\neq 3$, and the morphism $R_{ij}\to U_{C(\mathscr{T}_i)}\times_{\mathscr{K}_Q}U_{C(\mathscr{T}_j)}$ is defined by the isomorphism $\varphi_{\mathscr{T}_i\mathscr{T}_j}$. However, the isomorphism $\varphi_{\mathscr{T}_i\mathscr{T}_j}$ is defined by the difference $g^{ij}=g_{\Psi,\mathscr{T}_i}-g_{\Psi,\mathscr{T}_j}$. Over $U_{C(\mathscr{T}_1)}\times_UU_{C(\mathscr{T}_2)}\times_UU_{C(\mathscr{T}_3)}$, the cocycle condition
\[
\varphi_{\mathscr{T}_i,\mathscr{T}_k}=\varphi_{\mathscr{T}_j\mathscr{T}_k}\circ\varphi_{\mathscr{T}_i\mathscr{T}_j},
\]

is satisfied. This means exactly that the diagram above commutes in this case. 
\end{proof}

The collection $\underline{\mathscr{U}}=\{U_{C(\mathscr{T}_i)}\}$ is an \'{e}tale cover of $\mathscr{T}_Q$. For each $q\in Q(\mathbf{Q})$, the collection $\{g^{ij}(q)=g_{\Psi,\mathscr{T}_i}(q)-g_{\Psi,\mathscr{T}_j}(q)\}$ is a $1$-cocycle in $\check{C}^1(\underline{\mathscr{U}},\mathcal{O}^*)$. So it represents a line bundle $\mathcal{L}_q$. Lemma~\ref{cocycle condition} means  the gluing $\varphi_{\mathscr{T}_i\mathscr{T}_j}$ of the universal family is by twisting by the algebra of line bundles $\oplus\mathcal{L}_q$. 

\begin{proposition}\label{faithful}
Let $R':=U\times_{\mathscr{K}_Q}U$. The morphism $R\to R'$ is injective as a morphism between sheaves over the big \'{e}tale site.
\end{proposition}
\begin{proof}
Combine Proposition \ref{first faithful} and Proposition \ref{second faithful}. 
\end{proof}

\begin{proposition}
The $1$-morphisms $F:U\to \mathscr{K}_Q$ and $R\to U\times_{\mathscr{K}_Q}U$ induces a proper $1$-morphism $\overline{F}: \mathscr{T}_Q\to \mathscr{K}_Q$ between algebraic stacks. 
\end{proposition}
\begin{proof}
Let $R':=U\times_{\mathscr{K}_Q}U$. Consider the groupoid $(U, R', s',t', \mu',\epsilon')$ in $(\Sch/k)$ induced by the morphism $F:U\to \mathscr{K}_Q$. By Lemma \ref{cocycle condition}, we have a morphism $\overline{F}: (U,R, s,t, \mu,\epsilon)\to (U,R',s',t',\mu',\epsilon')$ between groupoids in $(\Sch/k)$, and it induces a morphism $\overline{F}: [U/R]\to [U/R']$ by the universal property of stackification. By (\cite{LMB} Proposition (3.8)), $[U/R']$ is a substack of $\mathscr{K}_Q$. The composition gives $\overline{F}: \mathscr{T}_Q\to \mathscr{K}_Q$.  Since $\mathscr{T}_Q$ and $\mathscr{K}_Q$ are both proper, $\overline{F}$ is proper by (\cite{ols13} Proposition 10.1.4 (iv)). 
\end{proof}

\begin{proposition}
The proper morphism $\overline{F}$ is surjective. 
\end{proposition}
\begin{proof}
Let $\mathscr{U}$ be the open dense locus of $\mathscr{K}_Q$ with trivial log structure. It is an open substack that classifying families all of whose geometric fibers are irreducible. The objects are polarized toric varieties $X_Q$ with a torus embedding $T\to X_Q$ constructed from the integral polytope $Q$. The moduli is only for the divisor $\Theta$. Fix any triangulation $\mathscr{T}$ and consider the open torus orbit of $U_{C(\mathscr{T})}$. This torus already parametrizes all pairs $(X_Q, \Theta)$ as above. Therefore, the image of $U_{C(\mathscr{T})}$ contains $\mathscr{U}$. Since $\overline{F}$ is proper and $\mathscr{U}$ is dense in $\mathscr{K}_Q$, $\overline{F}$ is surjective. In particular, $F: U\to \mathscr{K}_Q$ is surjective. By (\cite{LMB} Proposition (3.8)), $[U/R']$ is isomorphic to $\mathscr{K}_Q$. 
\end{proof}

\begin{proposition}
The $1$-morphism $\overline{F}: \mathscr{T}_Q\to \mathscr{K}_Q$ is representable. In particular $\overline{F}$ is proper as a representable morphism. 
\end{proposition}

\begin{proof}
By Proposition~\ref{faithful}, $R\to R'$ is injective. Then $\overline{F}$ is representable by (\cite{zhuav} Lemma 3.52). By (\cite{ols13} Proposition 10.1.2), for a representable separated morphism of finite type, the two properness mean the same.
\end{proof}

By construction, the algebraic stack $\mathscr{T}_Q$ is covered by open substacks $[U_{C(\mathscr{T})}/R_\mathscr{T}]$. The coarse moduli space $\mathcal{T}_Q$ is thus obtained by gluing the affine toric varieties $X_{C(\mathscr{T})}$ constructed by using the lattice $\mathbb{L}$. The gluing is by identifying different lattices in the same space $\mathbb{L}_\mathbf{Q}$. It is exactly the construction of the toric variety $X_{\Sigma(Q)}$ from the secondary fan $\Sigma(Q)$ and the lattice $\mathbb{L}^*$. Therefore $\mathcal{T}_Q\cong X_{\Sigma(Q)}$. Since $\mathscr{K}_Q$ is isomorphic to the main irreducible component of $\mathscr{TP}^{fr}[Q]$ introduced in \cite{Alex02} (\cite{ols08} Theorem 3.7.3), and the normalization of the irreducible main component of $\mathscr{TP}^{fr}[Q]$ is isomorphic to $X_{\Sigma(Q)}$ (\cite{Alex02} Corollary 2.12.3). The morphism induced by $\overline{F}$ between coarse moduli spaces is the normalization.

\begin{remark}
Strictly speaking, we need to check that the morphism between coarse moduli spaces is the compositions of the isomorphisms mentioned above. It can be done via the explicit description of the morphism in the proof of (\cite{Alex02} Theorem 2.11.8).
\end{remark}

\begin{theorem}[The Main Theorem]\label{compactification for GKZ}
Assume the lattice polytope $Q$ contains a regular simplex. We have a compatification $\mathscr{T}_Q$ of the moduli space of toric pairs $(X_Q,\Theta)$. Over $\mathbf{Q}$, $\mathscr{T}_Q$ is a proper Deligne-Mumford stack of finite type, with finite diagonal. It admits a coarse moduli space $\mathcal{T}_Q\cong X_{\Sigma(Q)}$. Furthermore, there is a proper, surjective, representable $1$-morphism $\overline{F}$ from $\mathscr{T}_Q$ to the moduli stack $\mathscr{K}_Q$ defined in \cite{ols08}, that induces the normalization between the coarse moduli spaces. 
\end{theorem}

\begin{remark}
The construction of $\mathscr{T}_Q$ can be carried out over $\mathbf{Z}$. In that case, we only get an Artin stack with finite diagonal. Our approach only produces normal coarse moduli spaces because our monoids are obtained from the Mori fan and are natually toric monoids. 
\end{remark}

\begin{remark}
There is another interpretation of the moduli space in terms of algebraic cycles in $\mathbf{P}^{N-1}$. The integral polytope $Q$ defines a finite morphism $X_Q\to \mathbf{P}^{N-1}$, whose image $X'_Q$ is an algebraic cycle in $\mathbf{P}^{N-1}$. If the base is changed to $k=\mathbf{C}$, the action of the big torus $(\mathbf{C}^*)^{N}/\mathbf{C}^*$ on $\mathbf{P}^{N-1}$ induces an action on algebraic cycles. In \cite{Alex02}, Alexeev defines a map from $X_{\Sigma(Q)}$ to the parameter space of algebraic cycles induced by torus actions on $X'_Q$. See also (\cite{GKZ} Chapter 8 Theorem 3.2). 
\end{remark}

\bibliographystyle{amsalpha}
\bibliography{Bibliography}

\end{document}